\documentclass[10pt,a4paper]{article}
\usepackage[utf8]{inputenc}
\usepackage{amsmath}
\usepackage{amsfonts}
\usepackage{amssymb}
\usepackage{amsthm}
\usepackage{mathtools}
\usepackage{mathabx} %pour la triple norme
\usepackage{tikz}
\usepackage{accents}
\usepackage{stmaryrd} %% \llbracket
\usepackage{setspace}  
\usepackage[left=1.5cm,right=1.5cm,top=2cm,bottom=2cm]{geometry}
\author{Frédéric Valet}

\usepackage{glossaries}
\newacronym{ZK22}{$\text{ZK}_{2d}$}{}
\newacronym{ZK23}{$\text{mZK}$}{}
\newacronym{ZK32}{$\text{ZK}_{3d}$}{}

\newtheorem{lemm}{Lemma}
\newtheorem*{lemm*}{Lemma}
\newtheorem{theo}[lemm]{Theorem}
\newtheorem{prop}[lemm]{Proposition}
\newtheorem{rem}[lemm]{Remark}
\newtheorem{coro}[lemm]{Corollary}
\newtheorem{defi}[lemm]{Definition}

\newcommand{\e}{E}

\newcommand{\Lc}{\mathcal{L}}
\renewcommand{\v}{v}
\newcommand{\spann}{\text{Span}}
\renewcommand{\ker}{\mathrm{Ker}}
\newcommand{\phit}{\phi}
\newcommand{\psit}{\psi}

\newcommand{\htilde}{h}
\newcommand{\by}{\mathbf{y}}
\newcommand{\bx}{\mathbf{x}}
\newcommand{\bv}{\mathbf{v}}

\newcommand{\balpha}{\mathbf{\alpha}}
\newcommand{\bi}{\mathbf{i}}
\newcommand{\bj}{\mathbf{j}}
\newcommand{\beone}{\mathbf{e}_1}

\newcommand{\oned}{\llbracket 1; d\rrbracket}
\newcommand{\oneK}{\llbracket 1; K\rrbracket}

\title{Asymptotic K-soliton-like Solutions of the Zakharov-Kuznetsov type equations}
\date{\today}
\author{Frédéric Valet}
\AtEndDocument{\bigskip{ \normalsize
 \begin{center}
  \normalsize \textsc{Frédéric Valet} \par \small 
  Université de Strasbourg \par
  CNRS, IRMA UMR 7501 \par
  F-67000 Strasbourg, France \par  
  \texttt{valet@math.unistra.fr}
 \end{center}
}}

   \spacing{1.15}

\begin{document}

\maketitle
\let\thefootnote\relax
\let\thefootnote\relax

\begin{abstract}
We study here the Zakharov-Kuznetsov equation in dimension $2$ and $3$ and the modified Zakharov-Kuznetsov equation in dimension $2$. Those equations admit solitons,  characterized by their velocity and their shift. Given the parameters of $K$ solitons $R^k$ (with distinct velocities), we prove the existence and uniqueness of a multi-soliton $u$ such that
\[ \| u(t) - \sum_{k=1}^K R^k(t) \|_{H^1} \to 0 \quad \text{as} \quad t \to +\infty. \]
The convergence takes place in $H^s$ with an exponential rate for all $s \ge 0$. The construction is made by successive approximations of the multi-soliton. We use classical arguments to control of $H^1$-norms of the errors (inspired by Martel \cite{Mar}), and introduce a new ingredient for the control of the $H^s$-norm in dimension $d\geq 2$, by a technique close to monotonicity.
\end{abstract}

\section{Introduction}

In this paper we study different versions of the equation
\begin{align}\label{ZK}\tag{ZK}
\begin{cases}
\partial_t u+ \partial_1 \left( \Delta u + u^p \right)=0, \\
u(0,\bx)= u_0(\bx)
\end{cases}
\quad (t,\bx) \in \mathbb{R}\times \mathbb{R}^d, \quad u(t,\bx) \in \mathbb{R}, \quad u_0 \in H^1,
\end{align}
where the pair $(d,p)$ is:
\begin{itemize}
\item $(2,2)$ for the Zakharov-Kuznetsov equation in dimension $2$, with $\bx= (x_1,x_2)$; \hfill \gls{ZK22}
\item $(2,3)$ for the modified Zakharov-Kuznetsov equation, with $\bx= (x_1,x_2)$; \hfill \gls{ZK23}
\item $(3,2)$ for the Zakharov-Kuznetsov equation in dimension $3$, with $\bx=(x_1,x_2,x_3)$, \hfill \gls{ZK32}
\end{itemize}
$\partial_i$ is the derivative with respect to the $i^\text{th}$ space-coordinate, and $\Delta=\sum\limits_{i=1}^d \partial_i^2$ the Laplacian. 

A solution of (\ref{ZK}) enjoys two conserved quantities, at least formally, the mass and the energy:
\begin{align*}
M(u(t)):=\frac{1}{2} \int_{\mathbb{R}^d} \vert u(t,\bx) \vert^2 d\bx \quad \text{and} \quad E(u(t)):= \int_{\mathbb{R}^d} \left( \frac{1}{2} \vert \nabla u(t,\bx) \vert^2- \frac{1}{p+1} ( u(t,\bx) )^{p+1} \right) d\bx.
\end{align*}

Equation (\gls{ZK32}) was first propsed by Zakharov and Kuznetsov \cite{KZ} to describe the evolution of non-linear ion-accoustic waves in magnetized plasma. Equations (\gls{ZK22}) and (\gls{ZK32}) were derived from the Euler-Poisson system in dimension $d=2$ and $d=3$ by Lannes, Linares and Saut in \cite{LLS}; (\gls{ZK22}) was also derived from the Vlasov-Poisson system by Han-Kwan \cite{Han13}. Some physical considerations in \cite{BS} explain how (\gls{ZK23}) is derived. 

Regarding the Cauchy problem, Faminskii \cite{Fam} first proved that (\gls{ZK22}) was globally well-posed in $H^1(\mathbb{R}^2)$, and Molinet-Pilod in \cite{MP} improved this result to local well-posedness in $H^s(\mathbb{R}^2)$ with $s>\frac{1}{2}$. In dimension $d=3$, Ribaud-Vento \cite{RV12.1} proved local well posedness in $H^s(\mathbb{R}^3)$ for $s>1$, and then Molinet-Pilod \cite{MP} and independently in Grünrock-Herr \cite{GH} proved global well-posedness in the same spaces. For (\gls{ZK23}),  the first result of local well-posedness was done by Linares-Pastor \cite{LP09} for $s>3/4$ (see also \cite{LP11}), and then Ribaud-Vento \cite{RV12.2} improved it to $s >1/4$. In a recent work Kinoshita \cite{Kin} improved local well-posedness of (\gls{ZK32}) in $H^s$ for $s\geq \frac{1}{2}$. From a different perspective,  Bhattacharya-Farah-Roudenko \cite{BFR} showed that solutions to the focusing (\gls{ZK23}) in $H^s$ for $s\geq \frac{3}{4}$ are global, provided that the mass of the initial data is less than the mass of the ground state.

\bigskip

The following scaling transform leaves the set of solutions invariant:
\begin{align*}
u_\lambda (t,\bx)= \lambda^{\frac{1}{p-1}} u(\lambda^\frac{3}{2}	t, \lambda^\frac{1}{2} \bx).
\end{align*}
In particular, the $H^{s_c}$-norm is conserved under the scaling, with the critical scaling exponent $s_c:= \frac{d}{2}-\frac{2}{p-1}$ : $ \| u_\lambda(t) \|_{H^{s_c}}= \| u (t) \|_{H^{s_c}}$. In the studied cases, the equations are in the subcritical case for (\gls{ZK22}) and (\gls{ZK32}) and in the critical case for (\gls{ZK23}).

\subsection{Solitons and multi-solitons}

Nonlinear travelling waves are special solutions of these equations. Those fundamental objects keep their form along the time, and move at a velocity $c$ in one direction. Their existence is the result of a balance between the dispersive and non-linear parts of (\ref{ZK}). From \cite{CMPS}, it is proved that those objects exist only if they move along the first axis, and thus satisfy the elliptic equation:
\begin{align}\label{elliptic}
-cu +\Delta u +u^p=0.
\end{align}
Without further assumption, there exist many solutions in $H^1(\mathbb R^d)$ of this elliptic equation. We will consider here only the ground state $Q_c \in H^1(\mathbb R^d)$ which is radially symmetric and positive: up to  translation in space, and up to the sign for $p=3$, see \cite{K} and \cite{BL}, $Q_c$ is the unique positive solution to \eqref{elliptic}. $Q_c(\bx -ct \beone)$ is thus a solution to \eqref{ZK}, called a soliton.

Denoting $Q:=Q_1$, we observe that $Q_c(\bx)= c^{\frac{1}{p-1}} Q(\sqrt{c}\bx)$. Unlike the solitons of the Korteweg-de-Vries equation, $Q$ has no explicit expression. By ODE arguments one can still obtain decay at infinity:
\begin{align} \label{decroitexpo}
\forall \balpha \in \mathbb{N}^d, \quad \vert \partial_1^{\alpha_1}\cdots \partial_d^{\alpha_d} Q(\bx) \vert = \vert \partial^\balpha Q(\bx) \vert \leq K_1(\alpha) e^{- \vert \bx \vert},
\end{align}
and we obtain exponential decay for all $Q_c$ by scaling. 

Given $K$ velocities $c^k >0$, shifts $\by^k \in \mathbb{R}^d$ and signs $\sigma^k\in \{(\pm 1)^p\}$, denote $R^k$ the soliton
\begin{align*}
R^k(t,\bx):= \sigma^k Q_{c^k} (\bx-c^k t \beone -\by^k).
\end{align*}
Due to the non-linearity, a sum of solitons is no longer a solution. However, if the velocities $c^k$ are distincts, the solitons will decouple and interact weakly, and one wonders if there exist nonlinear solutions $u(t)$ which behave like the sum of the $R_k$ for large times: more precisely $u(t)$ should behave like
\begin{align*}
 u(t,\bx) \simeq R(t,\bx) \quad \text{where} \quad R(t):= \sum_{k=1}^K R^k(t).
\end{align*}
Such solutions are called (pure) multi-solitons:

\begin{defi}
A solution $u$ of \eqref{ZK} is called a (pure) multi-soliton (or $K$-soliton) if for some $T_0 \in \mathbb R$, $u \in \mathcal{C}([T_0,\infty),H^1)$ and it behaves at infinity as a sum of solitons with distinct speeds:
\begin{align}\label{defi_multi}
u \in \mathcal{C}([T_0,\infty),H^1), \quad \text{ and } \quad \lim_{t\rightarrow +\infty} \left\| u(t)-\sum_{k=1}^K Q_{c^k}(\bx- \bv^k(t)) \right\|_{H^1}=0, 
\end{align}
where $\bv^k(t) = c^k t\beone - \by^k $, for some $0< c^1 < \cdots < c^K$ and $(\by^k)_k\in (\mathbb{R}^d)^K$. 

We say that such a $u$ is a multi-soliton associated with the velocities $(c^k)_k$ and the shifts $(\by^k)_k$.
\end{defi}

Multi-soliton were first constructed for the (KdV) equation using the inverse scattering method, see \cite{Miu76}. In the non integrable context, multi-solitons in the sense of the definition (\ref{defi_multi}) were first built for the nonlinear Schrödinger (NLS) and generalized Korteweg-de Vries (gKdV) equations by Merle \cite{Mer90}, Martel \cite{Mar} and Martel-Merle \cite{MM} in the $L^2$ -critical and subcritical cases, and then for the supercritical cases by Côte-Martel-Merle \cite{CMM11}. Combet in \cite{Com10} gave a classification of multi-solitons for the $L^2$-supercritical (gKdV) equation. The study of multi-solitons was also developped for other dispersive equations, like the non-linear Klein-Gordon equation in \cite{CM14}, the water wave system in \cite{MRT13} and the damped Klein-Gordon equation \cite{Fei98,CMYZ}. 

The dynamics of multi-solitons was also studied, mainly for Korteweg-de Vries type equations. We refer to \cite{MMT} for a stability result. One can also ask about the behavior as $t \to -\infty$ (or minimal time of existence) of multi-solitons: this in particular requires to understand the collision of solitons. For the (KdV) equation, see \cite{Miu76}, the inverse scattering method give explicit formulas which shows that the collision are elastic (with just and explicit shift in space): the multi-solitons structure remains at $-\infty$. However, for the non integrable quartic (gKdV) equation, Martel and Merle proved in \cite{MM07} that the collisions are not fully elastic: a pure 2-soliton at $-\infty$ is no longer a 2-soliton at $+\infty$, and they are able to describe the defect of elasticity. 

In this article, we will only focus on multi-solitons for positive times.

\subsection{Main results} 

Our main result in this paper is to prove the existence and uniqueness of multi-solitons of \eqref{ZK} in the sense of Definition \eqref{defi_multi}.
\begin{theo}\label{th1}
Let $K\in \mathbb{N}^*$, $K$ distinct velocities $0<c^1<\cdots<c^K$ and $K$ shifts $(\by^k)_k$. There exists a multi-soliton of (\ref{ZK}) associated to those velocities and shifts, denoted by $R^*$ and defined for times in $[T_0,+\infty)$. It is unique in $H^1$ in the sense of (\ref{defi_multi}).  Furthermore, $R^* \in \mathcal C^\infty([T_0,+\infty) \times \mathbb R^d)$ and there exist a constant $\delta>0$, and for all $s \ge 1$, a constant $A_s$ such that:
\begin{align*}
 \forall t\geq T_0, \quad \left\| R^*(t)-R(t) \right\|_{H^s}\leq A_s e^{-\delta t}.
\end{align*}
\end{theo}

Let us introduce right now an important constant related to the interaction of the solitons:
\begin{align*}
\sigma_0 := \min (c^1, c^2-c^1, \cdots, c^K-c^{K-1}).
\end{align*} 
The constant $\delta$ in Theorem \ref{th1} depends on the different velocities: it can be chosen as $\delta \leq \frac{\sigma_0^{3/2}}{8}$. The time $T_0$ depends on the velocities and the shifts. Observe that none of them depend on the regularity index $s$.

\vspace{1cm}

%%%%%%%%%   Uniqueness  %%%%%%%%%%%%%

An important outcome of this paper is the proof of uniqueness of multi-solitons. As far as we can tell, it is the first time that this property is proved, besides the case of the (gKdV) equations, where it was proved in the $L^2$ subcritical and critical cases in Martel \cite{Mar}. Regarding the $L^2$ supercritical case of (gKdV), recall  that \cite{Com10} gave an exhaustive classification: multi-soliton form a $K$-parameter family (each soliton has one instability direction, which yields one degree of freedom in the multi-soliton). 

\bigskip

%%%%%%%%%   (mZK3D)   %%%%%%%%%%%%%

We do not consider the modified Zakahrov-Kuznetsov equation in dimension $3$ (mZK$_{\text{3d}}$) in Theorem \ref{th1} for different reasons. Let us recall that, even though (mZK$_{\text{3d}}$) is not supported by any physical model, as far as we know, it is an $L^2$ supercritical equation, which makes it interesting by itself. In view of the study in Côte-Martel-Merle \cite{CMM11} for the $L^2$ supercritical (NLS) and (gKdV) equations, the construction would require the use of a extra topological argument. Recalling \cite{Com10} for (gKdV),  one would no longer expect uniqueness of multi-solitons, but rather a classification into a $K$-parameter family. In order to keep this article a reasonable size, we do not raise this case and leave it for future works. We focus here only on $L^2$ subcritical and critical cases. 

\bigskip

%%%%%%%%%   Stability      %%%%%%%%%%%%%

One can naturally ask about stability properties of  multi-solitons. Let us recall that de Bouard \cite{deB} proved  orbital stability of a soliton in the $L^2$ subcritical cases  (\gls{ZK22}) and (\gls{ZK32}), and orbital instability in the $L^2$ supercritical cases. Côte-Muñoz-Pilod-Simpson in \cite{CMPS} strengthened this result to asymptotic stability of a soliton for (\gls{ZK22}): if the initial condition is $H^1$-close to a soliton, then it converges to a 1-soliton in $H^1(\{ \bx; x_1 >\beta t\})$ for some $\beta >0$ small. \cite{CMPS} also give a result of asymptotic stability in the same spirit, for the sum of decoupled solitons: in view of Theorem \ref{th1}, their result can naturally be interpreted as asymptotic stability in $H^1(\{ \bx; x_1 >\beta t\})$ of \emph{multi-solitons}. The corresponding results for  (\gls{ZK32}) remain open.

In the $L^2$ critical case (\gls{ZK23}), solitons are unstable and can lead to some kind of blow up (see Farah-Holmer-Roudenko \cite{FHR} and Farah-Holmer-Roudenko-Kay \cite{FHRY}).

\bigskip
%%%%%%%%%   High interaction      %%%%%%%%%%%%%

The multi-solitons constructed in Theorem \ref{th1} are the usual one: each soliton interacts weakly with the others solitons, so that the trajectories of their center is not affected. Recently, the interest grew for highly interacting multi-solitons, where the trajectories of the centers is dictated (at leading order) by the interaction with its neighbour solitons. This happens for example for (gKdV) when considering to two solitons with same mass (and so, same speed).

The first result concerning highly interacting multi-solitons is due to Martel-Raphaël in \cite{MR18}, for the non-linear Schrödinger equation (NLS): they constructed a solution which behaves as a the sum of $N$-solitons placed on the vertices of regular $N$-gone (this solution blows up in infinite time). Nguyen \cite{Ngu_17} studied the problem for (gKdV): he constructed, for the $L^2$ subcritical and supercritical cases, $2$-solitons with mass $c_1=c_2=1$, and center at $ t \pm \ln(\widetilde{c}t)) + O(1)$ (so that the distance between the centers is $2\ln(t) + O(1)$). A similar regime was studied for (NLS) in Nguyen \cite{Ngu_19}. We expect that an analoguous phenomenon can happen for \eqref{ZK}. 

\subsection{Outline and notations}

Recall the parameters of the $K$ solitons are given in the statement of Theorem \ref{th1}: the velocities \begin{align*}
0<c^1 < \cdots <c^k < \cdots < c^K,
\end{align*}
and the shifts $\by^k =  (y^k_1,y^k_2)$ for $d=2$, or $\by^k=(y^k_1,y^k_2,y^k_3)$ for $d=3$. We denote as before
\[ R^k(t,x) = \sigma^k Q_{c^k} (\bx-c^k t \beone -\by^k) \quad \text{and} \quad R(t) = \sum_{k=1}^K R^k(t). \]
The main steps of the proof of Theorem \ref{th1} are the following.

\bigskip

We consider a sequence of times $S_n$ tending to $+\infty$, and $u_n$ the sequence of solutions of (\ref{ZK}) with final data $u_n(S_n)=R(S_n)$. We prove first that the error $u_n(t)-R(t)$ is controlled uniformly in $n$ in $H^1$ on a time interval of the form $[T_0,S_n]$, where $T_0$ independent of $n$:
\begin{align*}
\forall t \in [T_0,S_n], \quad \| u_n(t) - R(t) \|_{H^1(\mathbb{R}^d)} \leq C e^{-\sigma t}.
\end{align*}
This result uses monotonicity typical to (KdV)-like equations, see \cite{Mar}, and the coercivity of the linearised operator around the sum $R(t)$ of $K$ solitons. The arguments are close to \cite{Mar}. This is done in Section  \ref{controlH1}.

\medskip

Then, we prove that the error is also controlled in $H^s$, for any integer $s \ge 2$. This requires a special attention due to space dimension $d \geq 2$. For (gKdV), the arguments developed in Martel \cite{Mar} are appropriate functionals of the derivatives of $u$ at level $H^s$  combined with a Grönwall. In dimension $d\geq 2$, the different combinations of derivatives due to the non-linearity can not be managed by such an algebraic argument: the terms left aside would make the Gronwall argument fails. To circumvent this requires a careful treatment, see (\ref{example}) for more precision. The point is that the variation of those terms can be dealt with by a monotonicity-type argument. This is done in Section \ref{controlHs}.

 We conclude, in Section \ref{constructionfinale}, the existence of a multi-soliton $R^*$ in $H^s$: we take as initial condition for $R^*$ a weak limit of $u_n(T_0)$, show that it is actually a strong limit in $L^2$, and so by interpolation in $H^s$ for any $s \in \mathbb R^+$, and then use continuity of the flow in $H^s$ for large $s$ to obtain the rate of convergence $\| R^*(t) - R(t) \|_{H^s}$.
 
 \medskip 

We finish with uniqueness of multi-solitons, in Section \ref{unicite}. For this, we compute the difference between an arbitrary multi-soliton, and $R^*$ that we just constructed. As this is the difference of two nonlinear solutions, we obtain better estimates than before, taking advantage of smoothness and exponential rate of convergence for $R^*$. In particular, some estimates are based on $L^\infty$-norm of the third derivative of the difference between the multisoliton $R^*$ and $R$, which implies a necessary a $H^{3+\frac{d}{2}^+}$ control of $R^*-R$. It means that we actually use the $H^5$ regularity of $R^*$ (see the inequality (\ref{4ieme_terme_fin})).

\section{Control of the $H^1(\mathbb{R}^d)$-norm of the error}\label{controlH1}

We consider an increasing sequence of times $(S_n)_n$ going to $+\infty$, and the backward solution $u_n$ of (\ref{ZK}) satisfying 
\begin{align*}
u_n(S_n)=R(S_n).
\end{align*}
As $R(S_n) \in H^\infty$, the solution $u_n$ is well defined on a time interval $(T_n^*;S_n]$, and for all $s \ge 0$, $u_n \in \mathcal C((T_n^*;S_n], H^s)$. 

Our main concern in this section is to prove the following bootstrap proposition which states that if one has a uniform small of the error $\| u_n(t) - R(t) \|_{H^1}$ on some interval of time, the error does actually decay with an exponential rate on the same interval of times.

\begin{prop}\label{propo5}
There exist constants $L>0$, $\alpha>0$, $A_1$ and a time $T_0>0$ all depending on $\sigma_0$ satisfying $A_1 e^{-\frac{1}{L}\frac{\sigma_0}{8} T_0} \leq \frac{\alpha}{2}$ such that the following hold. Let $t^* \in (\max(T_0,T_n^*),S_n)$. If 
\begin{align}\label{hypo}
\sup_{t \in [t^*;  S_n]} \| u_n(t) -R(t) \|_{H^1} \leq \alpha, 
\end{align} 
then
\begin{align} \label{decay}
\forall t \in [t^*; S_n], \quad \| u_n(t)-R(t) \|_{H^1} \leq A_1e^{-\frac{1}{L}\frac{\sigma_0}{8} t}.
\end{align}
\end{prop}

This section is decomposed as follow. We first recall some basic properties of \eqref{ZK} solitons. The next subsection introduces mass and energy related functional localized around each soliton. We then develop a modulation technique to decompose the solution into a  sum of modulated solitons and a remainder $w$ which satisfies orthogonality conditions. Third, we control the evolution of local masses and local energies by a monotonicity argument. Those local quantities are quadratic in the remainder term $w$ (at leading order): due to the orthogonality condition, and an Abel transform, they are coercive and yield the desired bootstrap bound (\ref{decay}).

\subsection{Bounds on the interaction and coercivity}

As mentioned above, the interaction of two solitons is weak. For example, for the solitons $R^1$ and $R^2$, denoting $x_0:= \frac{c^1 t+c^2 t}{2}$, we have
(using also \eqref{decroitexpo}), for $i_1, i_2 \in \llbracket 1;d \rrbracket$:
\begin{align}
\MoveEqLeft
\int_{\mathbb{R}^d} R_1 R_2 + \vert \partial_{i_1} R_1 \partial_{i_2} R_2 \vert \leq C_{c^1, c^2}  \iint_{x_1=-\infty}^{x_0} e^{-\sqrt{c^2} \vert \bx-c^2t\beone \vert} d\bx +C_{c^1, c^2} \iint_{x_1=x_0}^\infty e^{-\sqrt{c^1} \vert \bx -c^1t\beone \vert} d\bx \nonumber \\
	& \leq C_{c^1, c^2} e^{-\sqrt{\sigma_0} \frac{\vert c^1-c^2 \vert}{2}t}.\label{interaction}
\end{align}

We can also notice that, as $Q_c$ is radially symmetric, by the change of variable $x_{i_2}':=-x_{i_2}$, with $i_1 \neq i_2$, there hold:
\begin{align} \label{parity}
 \iint \partial_{i_1}Q_{c^k} \partial_{i_2} Q_{c^k}d\bx = \frac{1}{4} \iint (\partial_{i_1} Q_{c^k}+\partial_{i_2}Q_{c^k})^2 d\bx -\frac{1}{4}\iint (\partial_{i_1} Q_{c^k}-\partial_{i_2}Q_{c^k})^2 d\bx = 0.
\end{align}

The study of a solution close to a multisoliton brings us naturally to study the linearised operator $\Lc_c$ around a soliton:
\begin{align}\label{L_c}
\Lc_c(\eta):= -\Delta \eta + c\eta -pQ_c^{p-1} \eta.
\end{align}
We recall some well known properties on the operator $\Lc:=\Lc_1$, see \cite{Wei82}, \cite{Wei85}, \cite{Wei86}, \cite{K}, \cite{deB} and a good review in \cite{CMPS}.
\begin{prop}\label{propo_Z}
The self-adjoint operator $\Lc$ satisfies the following properties:
\begin{enumerate}
\item $\ker \Lc$ is generated by two directional derivatives of the soliton : $\ker \Lc= \spann \left\{ \partial_i Q, i \in \llbracket 1,d \rrbracket \right\}$.
\item $\Lc$ has a unique negative eigenvalue $-\lambda_0$ (with $\lambda_0 >0$) of multiplicity $1$, and the corresponding eigenvectors are engendered by a positive, radially symmetric function $Z$. For more convenience, we suppose $\| Z \|_{L^2}=1$. $Z$ also satisfies the exponential decay (\ref{decroitexpo}).
\item The essential spectrum of $\Lc$ is on the real axis, and strictly positive:
\begin{align*}
\sigma_{ess}= [\lambda_{ess}; +\infty), \quad \text{ with } \lambda_{ess} >0.
\end{align*}
\item Up to orthogonality conditions, the operator $\Lc$ is coercive:
\begin{itemize}
\item for (\gls{ZK22}) and (\gls{ZK32}), if $u\perp \partial_i Q$ for $i\in \llbracket 1,d \rrbracket$ and $u\perp Q$, then 
\begin{align*}
(\Lc u,u) \geq \lambda_{ess} \| u \|_{H^1}.
\end{align*}
\item for (\gls{ZK23}), if $u\perp \partial_i Q$ for $i\in \llbracket 1,d \rrbracket$, $u\perp Q$ and $u \perp Z$ then 
\begin{align*}
(\Lc u,u) \geq \lambda_{ess} \| u \|_{H^1}.
\end{align*}
\end{itemize}
\end{enumerate}
\end{prop}

A detailed proof of coercivity is given in appendix.

\subsection{Cut-off functions, adapted masses and energies}

As in \cite{Mar}, we use adequate cut-off function, defined on one direction, which will mainly see the mass of one soliton. Consider the function $\psi:\mathbb{R} \rightarrow \mathbb{R}$ defined by:
\begin{align*}
\psi(x_1):=\frac{2}{\pi} \arctan(e^{-\frac{1}{L}x_1}),
\end{align*}
with $\psi(-\infty)=1$ and $\psi(+\infty)=0$, and $L$ to define later. We obtain \begin{align} \label{majoration_psi}
\vert \psi'(x_1) \vert \leq \frac{1}{L}\psi(x_1), \quad \vert \psi ^{(3)}(x_1) \vert \leq \frac{1}{L^2}\vert \psi'(x_1)\vert.
\end{align}
As in \cite{Mar}, the use of a parameter $L$ will play a role in the study of the monotonicity of the mass and for the proof of uniqueness, where we need to control the third derivative by the first. In this article, we make the choice to also use this parameter to the condition of coercivity on each $H_k$, see (\ref{coercivite_pb_L}). Considering the moving solitons, we want to separate them at half of the distance between two following solitons:
\begin{align}\label{mk}
\forall 1 \leq k \leq K-1 \text{ , } m^k(t):= \frac{c^{k+1}+c^k}{2}t+\frac{y^{k+1}_1+y^k_1}{2}, \text{ and } \psi^k(t,x_1):=\psi(x_1-m^k(t)),
\end{align}
and each cut-off function $\phi^k$ isolates one soliton:
\begin{gather*}
\phi^1(t,\bx):=\psi^1(t,x_1) \text{ , } \phi^K (t,\bx):= 1-\psi^{K-1}(t,x_1) \text{ , } \\
\forall 2 \leq k \leq K-1 ,\quad  \phi^k(t,\bx):= \psi^k(t,x_1)-\psi^{k-1}(t,x_1).
\end{gather*}

We define the new masses and energies, which mainly focus on the $k$-th soliton:
\begin{align*}
M^k(t):= \int u(t)^2 \phi^k(t) d\bx \quad \text{and} \quad E^k(t):= \int \left( \frac{1}{2} \vert \nabla u(t) \vert^2- \frac{1}{p+1} u(t) ^{p+1} \right) \phi^k(t) d\bx.
\end{align*}

For each $k$, the sum of mass $\sum\limits_{l \leq k} M^l(t)$ enjoys an (almost) monotonicity property. For the energy, the same claim is not clear. We consider adequate modified energies instead, which are better behaved:
\begin{align}\label{modified_energy}
\widetilde{E^k}(t):= E^k(t) + \frac{\sigma_0}{4} M^k(t).
\end{align}

\begin{rem}
We could have used compactly supported cut-off functions,  as $\psi$ a real valued $C^3$-function satisfying, on $\mathbb{R}$:
\begin{align*}
\psi(x_1)=\left\{ \begin{array}{l}
1 \text{ on } ]-\infty,-1] \\
0 \text{ on } [1, \infty[
\end{array}\right. , \quad \psi'(x_1) \leq 0, \quad \phi^k(t,x) = \psi\left( \frac{x_1-m^{k}(t)}{L}\right).
\end{align*}
This is done for the Schrödinger equation in \cite{MM}. However, in order to prove uniqueness, we use monotonicity argument and this requires exponentially decaying cut-off functions.
\end{rem}

\subsection{Modulation}\label{modulation}

We now use a modulation technique to obtain orthogonality conditions on the remainder term. For a sake of clarity, we suppose in this part that all the solitons are positive, so $\sigma^k =1$; the proof is identical for different signs. For $\mathcal{C}^1$ functions $(\bx ^k(t))_{1 \leq k \leq K}$ and $\widetilde{c^k}(t)$, we denote a modulated soliton, a modulated eigenvector defined in Proposition \ref{propo_Z} and respectively the error by 
\begin{align*}
\widetilde{R^k}(t,\bx) & :=Q_{\widetilde{c^k}(t)}(\bx -\by^k -\widetilde{\bx^k}(t)), \\
\widetilde{Z^k}(t,\bx) &:= Z_{\widetilde{c^k}(t)}(\bx-\by^k-\widetilde{\bx^k}(t))
\end{align*} 
and
\begin{align*}
w(t,\bx) & := u(t,\bx) -\widetilde{R}(t,\bx).
\end{align*}
We will ask for two types of orthogonality conditions:
\begin{align}\label{ortho1}
 \forall k\in \oneK, \forall i \in \oned \text{, } \int_{\mathbb{R}^d} w(t,\bx) \partial_i \widetilde{R^k}(t,\bx)d\bx =0.
\end{align}
or
\begin{align}\label{ortho2}
\forall k\in \oneK, \int_{\mathbb{R}^d} w(t,\bx) \widetilde{Z^k}(t,\bx) d\bx=0.
\end{align}

\begin{prop}
There exist a time $T_1>0$, $\alpha>0$, $K_1>0$ such that the following holds. Let $S_n >T_1$ and an initial condition $u(S_n)$ such that the solution $u$ of (\ref{ZK}) satisfies, for some $t^* \in [T_1,S_n]$:
\begin{align*}
\sup_{t^* \leq t \leq S_n} \|u(t)- R(t)\|_{H^1} \leq \alpha.
\end{align*}
Then :
\begin{enumerate}
\item for the subcritical cases (\gls{ZK22}) and (\gls{ZK32}), there exist $K$ unique functions $(\widetilde{\bx^k}(t))_{1\leq k \leq K}$ in $\mathcal{C}^1([T_1,S_n], \mathbb{R}^d)$ and we take $\widetilde{c^k}(t):=c^k$, such that the error term $w(t,\bx)$ satisfies the orthogonality conditions (\ref{ortho1}),
\item for the critical case (\gls{ZK23}), there exist $2K$ unique functions $(\widetilde{\bx^k}(t))_{1\leq k \leq K}$ in $\mathcal{C}^1([T_1,S_n], \mathbb{R}^d)$ and $(\widetilde{c^k}(t))_{1 \leq k \leq K}$ in $\mathcal{C}^1([T_1,S_n], \mathbb{R})$, such that the error term $w(t,\bx)$ satisfies the orthogonality conditions (\ref{ortho1}) and (\ref{ortho2}).
\end{enumerate}
We obtain a bound on the $H^1$-norm of the error and on the derivatives of $(\widetilde{x^k_i})_{i,k}$ and of $(\widetilde{c^k})_k$, on $[t^*,S_n]$:
\begin{align}\label{proches2}
\sum_k \left\vert \widetilde{\bx^k}(t)-c^k t \beone \right\vert+ \sum_k \left\vert \widetilde{c^k}(t)-c^k t  \right\vert+ \| w(t) \|_{H^1} \leq K_1\left\| u(t)-R(t) \right\|_{H^1} ,
\end{align}
and
\begin{align}\label{proches}
\forall k , \quad \left\vert \widetilde{\bx^k}'-\widetilde{c^k} \beone \right\vert + \left\vert \widetilde{c^k}' \right\vert \leq K_1 \sum_{j=1}^K  \left( \int w^2(t) e^{-\sqrt{\sigma_0}\vert \bx-c^j t\beone \vert} \right)^\frac{1}{2}+K_1 e^{-\frac{1}{2}\sigma_0 \sqrt{\sigma_0}t}.
\end{align}
\end{prop}

\begin{proof}
Let prove the modulation lemma in the critical case (\gls{ZK23}). We use the classical implicit function theorem. To do this, we consider the function:
\begin{align*}
g:
\left( \left( \begin{array}{ccc}
\widetilde{x^1_1}  &, \cdots, & \widetilde{x^K_1} \\
\widetilde{x^1_2} &, \cdots, & \widetilde{x^K_2} \\
\widetilde{c^1} &, \cdots, & \widetilde{c^K}
\end{array}\right), u \right) \rightarrow \left( \begin{array}{ccc}
\int w \partial_1 \widetilde{R^1} &, \cdots , &\int w \partial_1 \widetilde{R^K} \\
\int w \partial_2 \widetilde{R^1} &, \cdots , &\int w \partial_2 \widetilde{R^K}\\
\int w \widetilde{Z^1} &, \cdots , &\int w \widetilde{Z^K}
\end{array}\right).
\end{align*}
The derivative of $g_{i_1,k_1}$ with respect to $\widetilde{x^{k_2}_{i_2}}$, with $k_1 \neq k_2$, is an integral of the product of the derivatives of two solitons, considered far enough from each other. The terms along the diagonal of the differential of $g$ will be dominant:
\begin{align*}
\frac{d}{d\widetilde{x_i^k}}g_{i,k} = \int w \partial_i^2 \widetilde{R^k} - \| \partial_iQ_{c^k} \|_{L^2}^2.
\end{align*}
We need to ascertain that the crossing derivatives of one soliton are small enough, with $i_1 \neq i_2$ in $\oned$, by (\ref{parity}) and by parity:
\begin{align*}
\frac{d}{d\widetilde{x_{i_2}^k}}g_{i_1,k} = \int w \partial_{i_1} \partial_{i_2}\widetilde{R^k}- \int \partial_{i_2}Q_{\widetilde{c^k}} \partial_{i_1} Q_{\widetilde{c^k}}= \int w \partial_{i_1} \partial_{i_2}\widetilde{R^k} \quad \text{and} \quad \frac{d}{d\widetilde{c^k}} g_{i_1,k}=\int w \Lambda \widetilde{R^k}.
\end{align*} 
For the derivatives of the third components of $g$:
\begin{align*}
\frac{d}{d\widetilde{x^k_{i_2}}}g_{3,k}= \int w \partial_{i_2} \widetilde{Z^k} \quad \text{and} \quad \frac{d}{d\widetilde{c^k}}g_{3,k}=  \int w \Lambda \widetilde{Z^k} - \int \Lambda \widetilde{R^k} \widetilde{Z^k}=  \int w \Lambda \widetilde{Z^k} -\frac{1}{2 \lambda_0} \int Q Z.
\end{align*}

$Q$ and $Z$ are both positive, so the last term is negative.

Finally the derivatives of $g$ with respect to the variables $(\widetilde{x_i^k})_{i,k}$ and $(\widetilde{c^k})_k$ has a dominant diagonal. This proves that $\nabla_{(\widetilde{x_i^k})_{i,k}, (\widetilde{c^k})_k} g$ is invertible. We can then use the implicit function theorem: it gives us continuity and derivability of the variables along $u$, see \cite{MMT}. We also get an upper bound on $u-\widetilde{R}$:
\begin{align*}
\left\| (u-\widetilde{R})(t) \right\|_{H^1} 
	& \leq \| (u-R)(t)\|_{H^1} + \sum_{k=1}^K \left\| Q_{c^k}\left( \cdot - \by^k-c^kt\beone \right) - Q_{\widetilde{c^k}} ( \cdot - \by^k -\widetilde{\bx^k}(t)) \right\|_{H^1} \\
	& \leq \| (u-R)(t)\|_{H^1}+ \sum_k \vert \widetilde{\bx^k}(t)-c^k t\beone \vert + \sum_k \vert c^k t- \widetilde{c^k}(t)\vert.  \\
	& \leq C  \| (u-R)(t)\|_{H^1}.
\end{align*}

Now, it suffices to take the scalar product of the equation of the error:
\begin{align*}
d_t w + \partial_1 (\Delta w)= - \partial_1 \left( (\widetilde{R}+w )^p- \sum_{k=1}^K \widetilde{R^k}^p \right) + \sum_{k=1}^K  \left(\widetilde{\bx^k}'(t)-\widetilde{c^k} \beone \right) \cdot \nabla \widetilde{R}^k - \widetilde{c^k}'(t)\Lambda \widetilde{R^k}
\end{align*}
by $\widetilde{Z^j}$  with the orthogonality conditions $\frac{d}{dt}\left( \left\langle w, \widetilde{Z^j} \right\rangle \right)=0$ :
\begin{align*}
\MoveEqLeft
\sum_{k=1}^K \widetilde{c^k}'(t)\langle \Lambda\widetilde{R^k} , \widetilde{Z^j} \rangle- \sum_{k\neq j} ( \widetilde{\bx^k}'(t) -\widetilde{c^k} \beone ) \cdot \langle \nabla\widetilde{R^k}, \widetilde{Z^j} \rangle \\
	& = \left\langle w, \widetilde{c^j}' \Lambda \widetilde{Z^j}+\widetilde{ \bx^j}' \cdot \nabla \widetilde{Z^j} -\partial_1 \Delta \left( \widetilde{Z^j}\right) \right\rangle-\left\langle \partial_1 \left( (\widetilde{R}+w)^p - \sum_{k=1}^K \widetilde{R^k} ^p  \right), \widetilde{Z^j} \right\rangle,
\end{align*}
and similarly by taking the scalar product with $\partial_{i} \widetilde{R^j}$:
\begin{align*}
\MoveEqLeft
\sum_{k=1}^K ( \widetilde{\bx^k}'(t) -\widetilde{c^k} \beone ) \cdot \langle \nabla\widetilde{R^k}, \partial_i \widetilde{R^j} \rangle -\left(\widetilde{\bx^j}' -\widetilde{c^j} \beone \right) \cdot \left\langle w, \nabla \partial_i \widetilde{R^j} \right\rangle- \sum_{k\neq j} \widetilde{c^k}'(t)\langle \Lambda\widetilde{R^k} , \partial_1 \widetilde{R^1} \rangle +\widetilde{c^j}' \left\langle w , \partial_i \Lambda \widetilde{R^j} \right\rangle \\
	& = \left\langle w, \Lambda \widetilde{R^j} \right\rangle - \left\langle \left(\widetilde{R}+w \right)^p -\sum_{k=1}^K \widetilde{R^k}^p  -pw \widetilde{R^j}^{p-1}, \partial_1\partial_i \widetilde{R^j} \right\rangle.
\end{align*}

By taking the inverse of the system, we obtain the estimate of the proposition. One can notice that in the right side of the second equation, the term $\widetilde{\bx^1}'$ can be transformed into $\widetilde{\bx^1}'-\widetilde{c^k} \beone +\widetilde{c^k} \beone$, and the main terms with $\widetilde{\bx^1}'-\widetilde{c^k} \beone$ comes form the scalar product of $\langle \partial_i\widetilde{R^k}, \partial_1 \widetilde{R^1} \rangle$, decreasing $\alpha$ if necessary.
\end{proof}

We now go back to the proposition: once $T_1$ is fixed, if the distance on $[T_1,S_n]$ between $u$ and the sum of the $n$ decoupled solitons is less than $\alpha$, then we can modulate on the whole time inverval $[T_1,S_n]$.

\subsection{Evolution of the masses and energies}

We can assume that $L \ge 1$ ; in the following, the interaction between the different solitons will decay faster than the interaction of a soliton $\widetilde{R^k}$ with a weight function $\phi^{k'}$, and so will be neglected.

\begin{prop}
The following expansion of mass and energy holds for all time $t \in [t^*,S_n]$:
\begin{align} \label{estimeemasse}
\left\vert M^k(t) -\left( \int Q_{c^k}^2 +2 \int w \widetilde{R^k} \phi^k +\int w^2 \phi^k \right)\right\vert \leq  C\sqrt{L} e^{-\frac{1}{L} \frac{ \sigma_0}{4}t}(1+\alpha),
\end{align}
\begin{align}\label{estimeeenergie}
\left\vert E^k(t) - \left( E(Q_{c^k}) -c^k \int w \widetilde{R^k} \phi^k + \frac{1}{2}\int \left( \vert \nabla w \vert^2 -p \widetilde{R^k}^{p-1} w^2 \right) \phi^k  \right) \right\vert \lesssim \sqrt{L} e^{-\frac{1}{L}\frac{\sigma_0}{4} t} + \left( \sqrt{L} e^{-\frac{1}{L}\frac{\sigma_0}{4} t} + \alpha \right) \| w \|_{H^1}^2,
\end{align}
\begin{align}\label{local}
\left\vert \left( E^k(t) +\frac{c^k}{2} M^k(t) \right) -\left( E(Q_{c^k})+ \frac{c^k}{2} \int Q_{c^k}^2 \right) -\frac{1}{2} H_k(t) \right\vert \lesssim \sqrt{L} e^{-\frac{1}{L}\frac{\sigma_0}{4} t} + \left( \sqrt{L} e^{-\frac{1}{L}\frac{\sigma_0}{4} t} + \alpha \right) \| w \|_{H^1}^2,
\end{align}
with
\begin{align*}
H_k(t) :=\int \left( c^k w^2 + \vert \nabla w\vert^2  -p \widetilde{R^k}^{p-1} w^2 \right) \phi^k.
\end{align*}
\end{prop}

Observe that, for $k \neq 1$:
\begin{align*}
\int (\vert \widetilde{R^1} \vert + \vert \nabla \widetilde{R^1} \vert ) \phi^k \lesssim  e^{-\frac{1}{2}\sigma_0 t} + \sqrt{L} e^{-\frac{1}{L} \frac{ \sigma_0}{4}t} \text{ and } \int \widetilde{R^1}(1-\phi^1) \lesssim e^{-\frac{1}{2}\sigma_0 t} +  \sqrt{L} e^{-\frac{1}{L} \frac{ \sigma_0}{4}t}.
\end{align*}
This can be seen by considering the mass of $\phi^k$ is mainly far from the one of $\widetilde{R^1}$, cutting for example at the middle between $m^1$ and the center of the soliton $\widetilde{R^1}$. This estimate is also true for the other solitons $\widetilde{R^k}$.

\begin{proof}
The first estimate comes from a classical development of the mass; the second estimate:
\begin{align*}
\MoveEqLeft
\left\vert E^k(t) - \left( E(Q_{c^k}) + \int \left( \nabla \widetilde{R} \cdot \nabla w - w \vert \widetilde{R} \vert^p \right) \phi^k +\int \frac{1}{2} \nabla w \cdot \nabla w -\frac{1}{p+1}\left( (\widetilde{R} +w)^{p+1}- \widetilde{R}^{p+1} -(p+1)w \widetilde{R}^p \right)\phi^k  \right) \right\vert \\
	& \leq C e^{-\frac{1}{2}\sigma_0 \sqrt{\sigma_0}t} + C \int w^2 e^{- \frac{1}{2} \sigma_0 \sqrt{\sigma_0}t}\phi^k.
\end{align*}
The only terms at which we should take care about to prove (\ref{estimeeenergie}) are the higher power of $w$, for $3 \leq j \leq p+1$, and we use as in \cite{CMPS} the Sobolev embeddings $H^1(\mathbb{R}^2) \hookrightarrow L^q(\mathbb{R}^2)$ for $q\geq 2$ and $H^1(\mathbb{R}^3)\hookrightarrow L^3(\mathbb{R}^3)$:
\begin{align*}
\left\vert \int w^j \phi^k \right\vert \leq C \| w \|_{H^1}^j.
\end{align*}

\end{proof}

Now that we have obtained an estimate on the linearised operators $H_k$ around each of the solitons, we need to bound the evolutions of the energies and the masses. The following lemma expresses the (almost) monotonicity of the different quantities. Some estimates mimic those obtained in \cite{CMPS}.

\begin{lemm} \label{lem:mono}
There exists $L>0$ and $T_2>0$ large enough, such that for all $t^*>T_2$, the following evolutions of the mass and modified energy hold, for all $\kappa \in \llbracket 1,K \rrbracket$, $t \in [t^*, S_n]$:
\begin{align} \label{mass}
\sum_{k=1}^\kappa \left( M^k(S_n)- M^k(t) \right) \geq - C L e^{-\frac{1}{L} \frac{\sigma_0}{4} t}, \\
\sum_{k=1}^\kappa \left( \widetilde{E^k} (S_n) -\widetilde{E^k}(t) \right) \geq - C L e^{-\frac{1}{L} \frac{\sigma_0}{4} t}. \label{energy}
\end{align}
\end{lemm}

\begin{proof}
By computing the evolution of the mass for $\kappa \leq K-1$, the definition of $\sigma_0$, the decreasing function $\psi$, (\ref{majoration_psi}) and $\frac{1}{L^2} \leq \frac{\sigma_0}{4}$:
\begin{align}
\frac{d}{dt} \left( \sum_{k=1}^\kappa M^k \right)
	& =\int \left(  \frac{2p}{p+1} u^{p+1} - \vert \nabla u \vert^2 -2 (\partial_1 u)^2 -\frac{c^{\kappa +1}+c^\kappa}{2}u^2 \right) \partial_1 \psi^\kappa +\int u^2 \partial_1^3 \psi^\kappa \label{derivative_mass}\\
	& \geq \int \left( -\frac{2p}{p+1}u^{p+1} + \vert \nabla u \vert^2 +2 (\partial_1 u)^2 + 3 \frac{\sigma_0}{4} u^2 \right) \vert \partial_1 \psi^\kappa \vert . \nonumber
\end{align}

Only the term $u^{p+1}$ due to the non-linearity is non-positive. We deal with it by separating the integral into different pieces, according to the position of $x_1$. We use a parameter $R_0$ to adapt later, and an interval centred in $m^\kappa$: $I_t^\kappa=[c^\kappa t +y_1^\kappa+R_0 ; c^{\kappa+1} t +y_1^{\kappa+1} -R_0]$.

Firstly, consider that we are far from the variation of $\psi$,  so $x_1 \in {I_t^\kappa}^C$, we obtain :
\begin{align*}
\left\vert \psi'(x_1-m^{\kappa}(t) ) \right\vert \leq \frac{1}{L} C_0 e^{-\frac{1}{L}\frac{\sigma_0}{4} t}
\end{align*}
with the condition $T_2 \geq \frac{2}{\sigma_0}\sup \limits_{\kappa} \left\vert R_0 -\frac{1}{2}\left( y_1^{\kappa+1}-y_1^\kappa\right) \right\vert$. Moreover, the assumption (\ref{hypo}), and the Sobolev embeddings $H^1(\mathbb{R}^d) \hookrightarrow L^{p+1}(\mathbb{R}^d)$, for $(d,p)=(2,2)$, $(2,3)$ or $(3,2)$ with a constant $C_d$ :
\begin{align*}
\int_{{I_t^\kappa}^C}-\frac{2p}{p+1}u^{p+1}(t)\vert \partial_1 \psi^\kappa (t) \vert \geq -\frac{2p}{p+1} C_d^{p+1} \left( \| R(t) \|_{H^1} + \alpha \right)^{p+1} C_0\frac{1}{L} e^{-\frac{1}{L}\frac{\sigma_0}{4} t}.
\end{align*}

Secondly, the part where $\partial_1 \psi^\kappa$ can not come to the rescue uses that $u$ is close to $R$, and $R$ collapses on $I_t^\kappa$:
\begin{align*}
\MoveEqLeft
\int_{I_t^\kappa} -\frac{2p}{p+1} u^{p+1} \vert \partial_1 \psi^\kappa \vert +\int \vert \nabla u \vert^2 \vert \partial_1 \psi^\kappa \vert +\frac{3}{4} \sigma_0 u^2 \vert \partial_1 \psi^\kappa \vert  \\
	& \geq- \frac{2p}{p+1}C_d^{p+1} \left( \|u-R \|_{H^1} + \| R \|_{H^1(x_1 \in I_t^\kappa)} \right) ^{p-1} \left\| u \vert \partial_1 \psi^\kappa \vert^\frac{1}{2} \right\|_{H^1}^{2} +\int \vert \nabla u \vert^2 \vert \partial_1 \psi^\kappa \vert +\frac{3}{4} \sigma_0  \int u^2 \vert \partial_1 \psi^\kappa \vert \\
	& \geq \left( - \frac{2p}{p+1}C_d^{p+1} \left( \alpha+ \| R \|_{H^1(x_1 \in I_t^\kappa)} \right)^{p-1} +\min \left(1, \frac{3}{4} \sigma_0 \right) \right) \int \left( \vert \nabla u \vert^2 + u^2 \right) \vert \partial_1 \psi^\kappa \vert.
\end{align*}
By choosing $\alpha$ small enough depending on $\sigma_0$, and $T_2$ and $R_0$ large enough such that the constant on the right hand side is positive, we thus obtain:
\begin{align*}
\int \left( -\frac{2p}{p+1}u^{p+1} + \vert \nabla u \vert^2 +2 (\partial_1 u)^2 + 3 \frac{\sigma_0}{4} u^2 \right) \vert \partial_1 \psi^\kappa\vert  \geq -Ce^{-\frac{1}{L} \frac{1}{4}\sigma_0 t}.
\end{align*}
An integration from $t$ to $S_n$ concludes the first estimate on the mass.

Now, we compute the derivative of the sum of energy, and we remember that $-\partial_1 \psi^k >0$:
\begin{align}
\frac{d}{dt}\sum_{k=1}^\kappa E^k & = - \frac{1}{2} \int \left( \Delta u + u^p \right)^2 \partial_1 \psi^\kappa - \int \vert \partial_1 \nabla u \vert ^2 \partial_1 \psi^\kappa  + \frac{1}{2} \int \vert \partial_1 u \vert ^2 \partial_1^3 \psi^\kappa \nonumber  \\
		&  \quad \quad + p \int \vert \partial_1 u \vert ^2 u^{p-1} \partial_1 \psi^\kappa  + \int \left( \frac{1}{2} \vert \nabla u \vert^2 - \frac{1}{p+1} \vert u \vert^{p+1} \right) \frac{d}{dt}\psi^\kappa \label{derivative_eie} \\
	& \geq -\frac{1}{2L^2} \int \vert \partial_1 u \vert^2 \vert \partial_1 \psi^\kappa \vert - p \int \vert \partial_1 u \vert^2 \vert u \vert^{p-1} \vert \partial_1 \psi^\kappa \vert -\frac{1}{p+1} \frac{c^\kappa+ c^{\kappa-1}}{2}  \int \vert u \vert^{p+1} \vert \partial_1 \psi^\kappa \vert. \nonumber
\end{align}
We then use the same tools for $\widetilde{E^k}$ that we used for the mass. The first term is compensate by the ones of mass, while $L$ is large enough. The second is dealt again by considering separately $I_t^\kappa$ and ${I_t^\kappa}^C$:
\begin{align*}
\int_{{I_t^k}^C} \vert \partial_1 u(t) \vert^2 \vert u(t) \vert^{p-1} \vert \partial_1 \psi^\kappa(t) \vert \leq \| u (t)\|_{L^\infty}^{p-1} \left( \alpha^2 + \int \left( \partial_1 R(t) \right)^2 \right)\frac{1}{L}e^{-\frac{1}{L}\frac{\sigma_0}{4}t} \leq C \frac{1}{L}e^{-\frac{1}{L}\frac{\sigma_0}{4}t},
\end{align*}
and
\begin{align*}
\int_{I_t^k} \vert \partial_1 u \vert^2 \vert u \vert^{p-1} \vert \partial_1 \psi^\kappa \vert \leq C_d^{p-1} \| u \|_{H^2(I_t^k)}^{p-1} \int \vert \partial_1 u \vert^2 \vert \partial_1 \psi^\kappa \vert.
\end{align*}
Increasing $R_0$ if necessary to diminish $\| u \|_{H^2(I_t^k)}$, this term will also be compensated by the mass. The last one has already been done and uses exactly the same tools. This concludes the estimate for $\widetilde{E^k}$.
\end{proof}

\subsection{Proof of the exponential decay (\ref{decay})}

Recall the bound (\ref{hypo}) for $t \in [t^*,S_n]$:
\begin{align*}
\| u(t) - R(t) \|_{H^1} \leq \alpha,
\end{align*}
and recall the expression of the linearised operator localised around each soliton:
\begin{align*}
H_k(t) := \int \left( \vert \nabla w \vert^2 -p \widetilde{R^k}^{p-1} w^2 +c^kw^2 \right) \phi^k.
\end{align*}

Using the local evolutions of mass and energy (\ref{local}), we obtain for any sequence $(b_k)$ of scalar:
\begin{align}
\MoveEqLeft
\left\vert \sum_{k=1}^K b^k \left( E^k(t)+\frac{c^k}{2}M^k(t) \right) -\sum_{k=1}^K b^k \left( E^k(S_n)+ \frac{c^k}{2}M^k(S_n)\right) - \frac{1}{2}\sum_{k=1}^K b^k H_k(t) \right\vert \nonumber\\
	& \lesssim \sqrt{L} e^{-\frac{1}{L}\frac{\sigma_0}{4} t} + \left( \sqrt{L} e^{-\frac{1}{L}\frac{\sigma_0}{4} t} + \alpha \right) \| w \|_{H^1}^2.\label{global}
\end{align}

Let us motivate our choice of the coefficients $b^k:=\frac{1}{(c^k)^2}$. In order to use the mass monotonicity property, we develop the sum:
\begin{align*}
\sum_{k=1}^K \frac{b^kc^k}{2}M^k(t) = b^Kc^K \sum_{k=1}^K \frac{M^k(t)}{2} + \sum_{\kappa =2}^K \left( b^{\kappa-1} c^{\kappa-1} - b^{\kappa}c^{\kappa} \right) \sum_{k=1}^{\kappa-1} \frac{M^{k}(t)}{2},
\end{align*}
and for it to apply, each term has to come with a non negative coefficient: we see the constraint of $(b^kc^k)_k$ to be non increasing. We could choose, for example, $b^k=\frac{1}{c^k}$. Then, we obtain by the same formula a decomposition of the energy : $\sum\limits_k b^k\left(E^k(t)-E^k(S_n)\right)$. This is not suitable, because the energies $\sum \limits_{k=1}^\kappa E^k(t)$ do not satisfy the monotonicity property, as seen in Lemma \ref{lem:mono} above. To circumvent it, we wish to consider instead the modified energies $\widetilde{E^k}$ by adding some mass and we can now use the monotonicity of $\sum \limits_{k=1}^\kappa \widetilde{E^k}(t)$. It implies the constraint that the sequence $(b^kc^k)_k$ to be stricly decreasing, and our choice of $b^k:=\frac{1}{(c^k)^2}$ comes from there (recall that the $c^k$ are increasing order). We then obtain the following expansion:
\begin{align*}
\MoveEqLeft
\sum_{k=1}^K \frac{1}{(c^k)^2} \left( E^k + \frac{c^k}{2}M^k \right) \\
	&  =\frac{1}{(c^K)^2}\sum_{k=1}^K \widetilde{E^k} + \left( \frac{1}{2 c^K} -\frac{\sigma_0}{4}\frac{1}{(c^K)^2} \right)\sum_{k=1}^K M^k  \\
		& \quad + \sum_{\kappa=1}^K \left[ \left( \frac{1}{(c^\kappa)^2} -\frac{1}{(c^{\kappa+1})^2} \right) \sum_{k=1}^\kappa \widetilde{E^k} + \left( \frac{1}{2} \left( \frac{1}{c^\kappa} -\frac{1}{c^{\kappa+1}} \right) -\frac{\sigma_0}{4} \left( \frac{1}{(c^\kappa)^2}- \frac{1}{(c^{\kappa+1})^2} \right) \right) \sum_{k=1}^\kappa M^k \right].
\end{align*}

By the choice of $\frac{\sigma_0}{4}$ in $\widetilde{E^k}$, we observe that the coefficients in front of each sum of $\widetilde{E^k}$ and of $M^k$ are positive, we can thus use the monotonicity properties (\ref{mass}) and (\ref{energy}). (\ref{global}) gives us:
\begin{align}\label{grand}
\sum_{k=1}^K \frac{1}{(c^k)^2} H_k(t) \leq C \sqrt{L} e^{-\frac{1}{L}\frac{\sigma_0}{4} t} +C \left( \sqrt{L} e^{-\frac{1}{L}\frac{\sigma_0}{4} t} + \alpha \right) \| w \|_{H^1}^2
\end{align}

A lower bound of the left hand side uses the coercivity property, given in the following lemma, with the proof postponed to the appendix. It generalizes the classic coercivity of this quantity around one soliton.

\begin{lemm}\label{coercivity_lemm}
For subcritical and critical (\ref{ZK}), there exists $C_0 >0$, such that:
\begin{align}\label{coercivity}
C_0 \| w \|_{H^1}^2 - \frac{1}{C_0} \sum_{k=1}^K \left( \int \widetilde{R^k}w \right)^2 \leq \sum_{k=1}^K \frac{1}{(c^k)^2} H_k(t).
\end{align}
\end{lemm}

We now turn our attention to the terms $\int \widetilde{R^k} w$. The key to estimate this term is the appearance of a different sign in the development of the mass and the energy in (\ref{estimeemasse}) and in (\ref{estimeeenergie}). We then combine with the monotonicity. A first estimate comes from mixing (\ref{estimeemasse}) and (\ref{mass}):
\begin{align*}
2 \int w(t) \widetilde{R^1}(t)\phi^1(t) & \leq M_1(t) -M_1(S_n) - \int w^2(t) \phi^k(t) + C \sqrt{L}e^{-\frac{1}{L}\frac{\sigma_0}{4}t} \\
	& \leq \| w(t) \|_{H^1}^2 + CLe^{-\frac{1}{\sqrt{L}}\frac{\sigma_0}{4}t}.
\end{align*}

Because of the lack  of monotonicity for $E^1$, we use again the mass $M^1$ to find back $\widetilde{E^1}$, and we obtain the other inequality mixing (\ref{estimeeenergie}) and (\ref{energy}):
\begin{align*}
\left( c^1- \frac{\sigma_0}{8} \right) \int w(t) \widetilde{R^1}(t) \phi^1 (t) & \geq -\widetilde{E^1}(t) + \widetilde{E^1}(S_n) - C\| w(t) \|_{H^1}^2 \left( 1+ \sqrt{L}e^{-\frac{1}{L}\frac{\sigma_0}{4}t} \right) -C\sqrt{L} e^{-\frac{1}{L}\frac{\sigma_0}{4}t} \\
	& \geq- C\| w(t) \|_{H^1}^2 \left( 1+ \sqrt{L}e^{-\frac{1}{L}\frac{\sigma_0}{4}t} \right) -C\sqrt{L} e^{-\frac{1}{L}\frac{\sigma_0}{4}t}.
\end{align*}
For the other $\int w \widetilde{R^k}\phi^k$, we proceed by induction using again the monotonicity. We conclude then the estimates:
\begin{align}\label{une_direction}
\left\vert \int w(t) \widetilde{R^k}(t)\phi^k(t) \right\vert \lesssim \sqrt{L}e^{-\frac{1}{L}\frac{\sigma_0}{4}t} + \| w(t) \|_{H^1}^2.
\end{align}
(\ref{grand}) can now bound the $H^1$-norm of the error into:
\begin{align*}
\lambda_0 \| w(t) \|_{H^1}^2 \lesssim \sqrt{L}e^{-\frac{1}{L}\frac{\sigma_0}{4}t} +\left( \sqrt{L}e^{-\frac{1}{L} \frac{\sigma_0}{4}t} + \alpha \right) \|w(t) \|_{H^1}^2.
\end{align*}
Taking $T_0\geq \max (T_1,T_2)$ large enough, and $\alpha$ small enough, we infer
\[ \| w \|_{H^1} \lesssim \sqrt{L} e^{-\frac{1}{L} \frac{\sigma_0}{8}t}. \]
Now we can improve the bound on the parameter of modulation, by integrating (\ref{proches}) from $t$ to $S_n$, and $ \vert \widetilde{\bx^k}(S_n) -c^k S_n \beone \vert = \widetilde{c^k}(S_n) -c^k S_n =0$ :
\begin{align*}
\sum_{k=1}^K \vert \widetilde{\bx^k}(t) -c^k t \beone \vert + \vert \widetilde{c^k}(t) -c^k t \vert \lesssim C L^\frac{3}{2} e^{-\frac{1}{L}\frac{\sigma_0}{8}t}.
\end{align*}

We then conclude using the bounds from the modulation:
\begin{align*}
\| (u-R)(t) \|_{H^1} \leq \|w(t) \|_{H^1} + \| R(t)-\widetilde{R}(t) \|_{H^1} \leq C L^\frac{3}{2}e^{-\frac{1}{L}\frac{\sigma_0}{8} t}.
\end{align*}
By taking $A_1$ depending on $L$, and finally $T_0$ large enough, we have proved (\ref{decay}). This concludes the proof of Proposition \ref{propo5}.

\bigskip

A straightforward bootstrap argument shows that the minimal time $t^*$ down to which (\ref{decay}) holds is actually $t^*=\max(T_0,T_n^*)$, that is:

\begin{coro}
For all $t \in [(\max(T_0,T_n^*),S_n]$, there hold
\[ \| u(t) - R(t) \|_{H^1} \leq A_1e^{-\frac{1}{L}\frac{\sigma_0}{8} t}. \]
\end{coro}

\begin{proof}
This is a consequence of Proposition  \ref{propo5}, the final condition $u(S_n) - R(S_n)=0$, and a continuity argument ($u \in \mathcal C([t^*,S_n],H^1)$).
\end{proof}

We also obtain that the minimal time of existence $T_n^*$ is in fact lower than $T_0$:
\begin{coro}
The interval of existence of $u_n$ contains $[T_0,S_n]$, and $T_n^* <T_0$. Furthermore, $u_n$ is in $\mathcal{C}([T_0,S_n],H^s)$ for any $s\geq 1$.
\end{coro}
\begin{proof}
It is a consequence of the previous corollary and of the local well-posedness theory recalled in Appendix B. In fact, we get that:
\begin{align*}
\forall t \in [\max(T_0,T_n^*),S_n], \quad \|u_n(t) \|_{H^1} \leq A_1 e^{-\frac{1}{L} \frac{\sigma_0}{8}T_0} + \max_{T_0 \leq t \leq +\infty} \|R(t)\|_{H^1} =:C,
\end{align*}
and the theory of local well-posedness can apply at any point of the interval $[T_0,S_n]$, with a time of existence $T(C)$.
The regularity of $u_n$ comes from the initial data in $H^s$ for any $s\geq 1$, and the (local) continuity of the flow on the interval of existence.

\end{proof}

\section{$H^s$-estimate}\label{controlHs}

In the previous section, we proved that there exist a time $T_0$, constants $\delta_1 := \frac{1}{L}\frac{\sigma_0}{8}\geq \sqrt{\sigma_0}\frac{\sigma_0}{8}$ and $A_1$ such that :
\begin{align*}
\forall t \in [T_0, S_n], \quad \| u_n(t) -R(t) \|_{H^1} \leq A_1 e^{-\delta_1 t}.
\end{align*}

In this section, we want first to prove that the $H^s$-norm of the error is exponentially decreasing. All the $\dot{H}^s$-norm decay exponentially with the same rate $\delta_4$, but the constants involved highly depend on $s$.

\begin{prop}\label{propo_decroit_Hs}
Let $s\geq 4$. There exist constants $A_s$, $\delta_4:=\frac{\delta_1}{2}$ and a time $T_4\geq T_0$ such that the following bound holds : 
\begin{align}\label{borneHs}
\forall t \in [T_4;  S_n], \quad 
\| u_n(t)-R(t) \|_{\dot{H}^s} \leq A_s e^{-\delta_4 t}.
\end{align} 
\end{prop}

 The constants above are independant of $n$, and more importantly, the time on which the exponential decrease holds is independant of $s$. 
 
 Recall that the $u_n$ satisfy $u_n(S_n)=R(S_n)$, and $u_n$ is  defined on $[T_0,S_n]$. As above, we drop the index $n$ and denote $u_n$ by $u$; in the computations below, the constants involved will not depend on $n$ either. We denote $\v$ the difference of the solution $u$ with the sum of $K$ decoupled solitons by $\v$:
\begin{align}
\v(t,\bx):= u(t,\bx)- R(t,\bx).
\end{align}

Estimate \eqref{borneHs} was obtained by Martel \cite{Mar} for subcritical and critical (gKdV). Let us explain why this proposition requires a new argument in the context of \eqref{ZK}. 
The goal of this part is to obtain an inequality of the type:
\begin{align}\label{example}
\left\vert \frac{d}{dt}\left( \| v \|_{\dot{H}^s}^2 \right)+O(l.o.t.)  \right\vert \lesssim \| \v\|_{H^1} + \| \v \|_{\dot{H}^s}^{2-\epsilon} \| \v \|_{H^1}^{\epsilon} + \| \v \|_{\dot{H}^s}^3,
\end{align}
for some fixed $\epsilon>0$ (possibly small). The idea here is that we already know the exponential decay of the $H^1$-norm of the error $\v$, and we need by any means a bound better than $\| \v \|_{\dot{H}^s}^2$, like a power equal to $3$. In other words, an inequality which fails the estimate would be of the kind:
\begin{align}\label{problematic_situation}
\left\vert \frac{d}{dt} \left( \| u \|_{\dot{H}^s}^2\right) + O(\text{l.o.t.}) \right\vert \lesssim \| \v\|_{\dot{H}^s}^2 +O(\text{l.o.t.}),
\end{align}

 Let recall the strategy for $u$ a solution to the (KdV) equation in \cite{Mar}. A direct derivative of the $\dot{H}^2$-norm of the solution gives:
\begin{align*}
\frac{d}{dt} \left( \int \vert \partial_x^2 u \vert^2 \right) = -2 \int \partial_x^2 u \partial_x^3(u^2).
\end{align*}
This term is trilinear in $u$. We next replace $u$ by $\v+R$, and by developing those terms, we see some tricky terms appear:
\begin{align}\label{tricky_terms}
\left\vert \int \partial_x^2 \v \partial_x^2 \v \partial_x(R) \right\vert \lesssim \| \v \|_{\dot{H}^2}^2,
\end{align}
and prevent from achieving an inequality of the sense of (\ref{example}). These quadratic terms with maximal number of derivatives on $\v$ are precisely the ones that also prevent to construct a multi-solitons via fixed point argument using dispersive estimate on the flat space.

To get rid of those terms, one solution is to modify the functional and consider
\begin{align*}
\frac{d}{dt} \left( \int \vert \partial_x^2 u \vert^2 -c \int \partial_x u \partial_x u u \right) \lesssim ``n\text{-linear terms with lower derivatives and }n>p+1 \text{''}.
\end{align*}

This strategy works in dimension $d=1$. However, this technique does not apply anymore in dimension $d\geq 2$, let us see why for (\gls{ZK22}). Following the same scheme, the derivative of the $\dot{H}^2$-norm is:
\begin{align*}
\frac{d}{dt} \left( \sum_{i_1,i_2 =1}^d \int \vert \partial_{i_1}\partial_{i_2} u \vert^2 \right) = -2 \sum_{i_1,i_2 =1}^d \int \partial_{i_1,i_2} u \partial_{1,i_1,i_2} (u^2).
\end{align*}
By replacing all the $u$ by $\v+R$, we see that we need to compensate for the tricky terms  $\int (\partial_{i_1 i_2}\v)^2 \partial_1 R$ to avoid the problematic situation illustrated in (\ref{problematic_situation}). Trying to mimic the previous strategy, let us identify the different terms to add to modify the derived quantity :
\begin{itemize}
\item The crossed-derivatives terms: $\int \partial_1 u \partial_2 u u$. They are not useful for those combinations. In fact, the time derivative of this quantity will give terms with an odd number of derivative $\partial_2$, which does not correspond to our situation.
\item The terms with same derivatives: $\int \partial_1 u \partial_1 u u$ and $\int \partial_2 u \partial_2 u u$. However, no combination of those terms can cancel the trilinear term.
\end{itemize}

We thus need another method to deal with the trilinear terms, which is the purpose of our next result. The idea is the following. Observe that the tricky terms are localised around the center of each soliton, which recalls the situation of the derivative of the energies in (\ref{derivative_eie}), where the second derivative of $u(t)$ is localised around the main variation of $\psi^k$, between two successive solitons:
\begin{align*}
\frac{d}{dt}\sum_{k=1}^\kappa E^k(u)= -\frac{1}{2}\int \vert \Delta u \vert^2 \partial_1 \psi ^k + \text{``better behaved terms''}.
\end{align*}
In other words, the time derivative of a localised $H^{s-1}$ norm bounds a localised $H^s$-norm. After integration in time, we get
\[ \int_t^{S_n} \int \vert \Delta u \vert^2 \partial_1 \psi^k \lesssim E^k(u)(S_n) - E^k(u)(t) \lesssim \| v(t) \|_{H^1}^2. \]
Now, a similar bound can be obtained with $v$ instead of $u$, and from the previous section, we already know that $\| v(t) \|_{H^1}$ has an exponential decay rate. It remains to do one further observation: after performing the above computations, we now don't really need that $\psi^k$ has its variation localized away from the solitons (because the $H^1$ bound has already been obtained), and so we can center it around $R^k$. Summing up in $k$, we obtain exponential decay for
 \[ \int_t^{S_n} \int \vert \nabla^2 v \vert^2 \partial_1 (R^{p-1}), \] 
and this is exactly what is needed to adapt the strategy for (gKdV) to \eqref{ZK}.

The remainder of this section is dedicated to make the above outline rigourous, so as to complete the proof of Proposition \ref{propo_decroit_Hs}.

\bigskip

In dimension $d$, for $i\in \llbracket 1,d\rrbracket$, we denote the derivative of a function $f$ with respect to the $i^\text{th}$ variable by $\partial_i f$ or $f_i$. A multi-index is denoted by a bold letter $\bi$, with length $\vert \bi \vert$ and coefficients $\bi=(i_1, \cdots, i_{\vert \bi \vert})$. Naturally, if $\vert \bi \vert =k$ with $k\in \mathbb{N}$, then $f_\bi$ denotes the derivative in the following directions: $\partial_{i_1} \cdots \partial_{i_k} f$.

Let $A$ be a large enough constant to be chosen later, and which satisfies:
\begin{align*}
\forall k=1, \dots, K, \ \forall \bx \in \mathbb R^d, \quad \frac{A}{1+x_1^2} \geq \vert Q_{c^k} (\bx) \vert^{p-1}+\vert \nabla Q_{c^k} (\bx) \vert^{p-1}.
\end{align*}
We introduce now an adequate monotone function:
\begin{align}
\eta (t,\bx):= 1 + \sum_{k=1}^K \left( A\arctan( x_1 -c^k t -y_1^k)+\frac{\pi}{2} \right).
\end{align}
We call this function a threshold function along the first axis. It increases along $x_1$, and satisfy the fundamental pointwise estimate, at each $t$ and $i \in \llbracket 1, d \rrbracket$:
\begin{align*}
\eta_1(t) \geq \vert R(t) \vert^{p-1}+\vert \nabla R(t) \vert^{p-1}.
\end{align*} 

Due to the previous discussion, it is natural to define the function:
\begin{align}\label{G_s}
G_s(t):= e^{-\delta_1 t} + \| v(t) \|_{\dot{H}^s}^{2-\frac{1}{s}} \| v(t) \|_{H^1}^{\frac{1}{s}} + \|v(t) \|_{\dot{H}^s}^2 \|v(t)\|_{H^3}(1+\|v(t) \|_{H^1})^{p-2} .
\end{align}
This functional controls the interaction of different solitons, and the terms with the $H^1$-norm of the error:
\begin{align*}
e^{- \sqrt{\sigma_0}\frac{\sigma_0}{2}t} & \leq G_s(t), \\
\| v(t) \|_{H^1} & \leq A_1 G_s(t), \\
\| v(t) \|_{H^1}^{p+1} & \leq A_1^{p+1} G_s(t).
\end{align*}

We recall the useful Gagliardo-Niremberg interpolation, with $s' \leq n \leq s$, and $\theta= \frac{n-s'}{s-s'}$ :
\begin{align}\label{GN1}
\| f \|_{\dot{H}^n} \lesssim \| f \|_{\dot{H}^s}^\theta \| f \|_{\dot{H}^{s'}}^{1-\theta}.
\end{align}

Observe that the norms $\dot{H}^s$ with small $s$ are somehow more difficult to deal with: we deal with the nonlinearity using the Sobolev embedding $\dot{H}^{s}(\mathbb{R}^d) \hookrightarrow L^\infty(\mathbb{R}^d)$ which requires a high regularity index $s > d/2$. This is not difficult to overcome: we will prove (\ref{borneHs}) directly for $s\geq s_0$ large enough, and conclude by interpolation. For (\ref{ZK}), we can choose $s_0=4$.

\subsection{Control of a localised $\dot{H}^s$-norm of the error.}

In this section we establish a lemma to control the $\dot{H}^s$-norm of the error as (\ref{tricky_terms}), with a weight equal to the threshold function.

\begin{lemm} \label{lem:mono2}
Let $s \in \mathbb{N}_{\backslash \{0\}}$, a multi-index $\bi= (i_1, \cdots,i_{s-1})$, with $i_j \in \llbracket 1, d \rrbracket$. The following estimate holds, for $t\in [T_0,S_n]$:
\begin{align}
\int_t^{S_n} \int \sum_{l=1}^d \vert \v_{\bi l}(t',\bx) \vert^2 \eta_1(t',\bx) d\bx dt' \lesssim 
	& \int_t^{S_n} G_s(t') dt' + \|v(t)\|_{\dot{H}^{s-1}}^2. \label{equation10}
\end{align}
\end{lemm}

\begin{rem}
By summing over $\bi$, we obtain the $\dot{H}^s$-norm with the weight $\eta_1$.  Furthermore, we can notice that this bound is acceptable in view of \eqref{example}.
\end{rem}

\begin{proof}
The error satisfies the following equation:
\begin{align*}
\frac{d \v}{dt}=- \frac{d}{dt}R + \partial_1 \left( \Delta R +R^p \right) - \partial_1 \left( \Delta \v +(R+\v)^p- R^p \right).
\end{align*}
We differentiate $\bi$ times, multiply by $\v_{\bi}\eta$ and integrate in space to get:
\begin{align}
\MoveEqLeft
\sum_{l=1}^d \int \v_{\bi l}^2 \eta_1 + 2 \int (\v_{1\bi})^2 \eta_1 \nonumber \\
	& = \int \v_{\bi}^2 \eta_{111} - \frac{d}{dt}\left( \int \v_\bi^2 \eta \right) + \int \v_\bi^2 \frac{d\eta}{dt}+ 2(-1)^{\vert \bi \vert+1} \int \v\partial_\bi \left( \partial_\bi \left( \frac{dR}{dt} + \partial_1 \left( \Delta R +R^p \right) \right) \eta \right) \label{passe_encore}\\
	& \quad - 2 \int \v_\bi \partial_{1\bi} \left( pR^{p-1}\v \right) \eta - 2 \int \v_\bi \partial_{1i} \left((R+\v)^p- R^p -pR^{p-1}\v\right) \eta. \label{embetant}
\end{align}

The left side of this equation is the left side of (\ref{equation10}) up to a constant. Consider now the terms of (\ref{passe_encore}). We want to integrate them in time. We accept a bound by the $\dot{H}^{s-1}$-norm, because the function $\eta$ and its derivative are bounded. The last term of (\ref{passe_encore}) uses that for each $k$, $\partial_t R^k + \partial_1 (\Delta R^k + (R^k)^2)=0$, and the interactions between the different solitons are weak and exponentially decreasing in (\ref{interaction}). We obtain:
\begin{align*}
\MoveEqLeft
\int_t^{S_n} (\ref{passe_encore}) dt' \lesssim \| \v(t) \|_{\dot{H}^{s-1}}^2 + \int_t^{S_n} \| \v(t') \|_{\dot{H}^{s-1}}^2 + \| \v(t') \|_{L^2} e^{-\frac{\sigma_0}{4}\sqrt{\sigma_0}t'}dt'.
\end{align*}

The tricky term is (\ref{embetant}). The bilinear term in $\v$ gives us, by the embedding $H^2(\mathbb{R}^d) \hookrightarrow L^\infty(\mathbb{R}^d)$: 
\begin{align}
\left\vert \int \v_\bi \partial_{1\bi} (R^{p-1}\v)\eta \right\vert &= \left\vert \int v_\bi \partial_{1\bi} v R^{p-1}\eta + \int v_i \left( \partial_{1\bi} (R^{p-1}v)-\partial_{1\bi} (v) R^{p-1} \right) \eta \right\vert \nonumber \\
	& \lesssim \left\vert \frac{1}{2} \int \v_\bi ^2 \partial_1(R^{p-1}\eta) \right\vert + \| \v\|_{\dot{H}^{s-1}} \| \v\|_{H^{s-1}} \lesssim \| \v \|_{\dot{H}^{s-1}}\left( \| \v\|_{\dot{H}^{s-1}} + \|\v\|_{L^2} \right) \nonumber\\
	& \lesssim G_s(t).\label{majoration1}
\end{align}

Now the trilinear (or quadrilinear) term of (\ref{embetant}). We use the Cauchy-Schwarz inequality, the distribution of the derivatives on the different terms and the embedding $H^2(\mathbb{R}^d) \hookrightarrow L^\infty(\mathbb{R})$:
\begin{align*}
\left\vert \int \v_\bi(t) \partial_{1\bi}(\v^2(t)) \eta(t) \right\vert 
	& \lesssim \| \v(t) \|_{\dot{H}^{s-1}} \sum_{\bj \in (l,\bi), \vert \bj \vert \leq \lfloor \frac{s}{2} \rfloor} \| v(t) \|_{\dot{H}^{s- \vert \bj \vert}} \| v_\bj(t) \|_{L^\infty}  \\
	& \lesssim  \|v(t)\|_{\dot{H}^{s-1}} \sum_{\bj \in (l,\bi), \vert \bj \vert \leq \lfloor \frac{s}{2} \rfloor} \| v(t) \|_{\dot{H}^{s-\vert \bj \vert}} \| v(t) \|_{H^{\vert \bj \vert +2}}. 
\end{align*}
Using the interpolation (\ref{GN1}) with $s'=1$:
\begin{align}
\left\vert \int \v_\bi(t) \partial_{1\bi}(\v^2(t)) \eta(t) \right\vert 
	& \lesssim  \| v(t) \|_{\dot{H}^s}^{1-\frac{1}{s-1}} \| v(t) \|_{H^1}^{\frac{1}{s-1}} \sum_{\bj \in (l,\bi), \vert \bj \vert \leq \lfloor \frac{s}{2} \rfloor} \| v(t) \|_{\dot{H}^s}^{1-\frac{\vert \bj \vert}{s-1}} \| v(t) \|_{H^1}^{\frac{\vert \bj \vert}{s-1}} \left(\| v(t) \|_{H^1} +  \| v(t) \|_{\dot{H}^s}^{\frac{\vert \bj \vert+1}{s-1}} \| v(t) \|_{H^1}^{\frac{s-\vert \bj \vert-2}{s-1}} \right) \nonumber \\
	& \lesssim \| v(t) \|_{H^1}^3 + \|v(t) \|_{\dot{H}^s}^2 \| v(t) \|_{H^1} \lesssim G_s(t), \label{majoration2} \\
\left\vert \int \v_\bi (t) \partial_{1\bi}(\v^3(t)) \eta(t) \right\vert 
	& \lesssim \| \v(t) \|_{H^1}^4+ \| \v(t) \|_{\dot{H}^s}^{2} \|\v(t) \|_{H^1} \| \v(t) \|_{H^2} \lesssim G_s(t). \nonumber
\end{align}

By integrating (\ref{majoration1}) and (\ref{majoration2}) from $t$ to $S_n$, and the Gagliardo-Nirenberg interpolation (\ref{GN1}), we obtain the bounds (\ref{equation10}).
\end{proof}

\subsection{Control of the variation of the $\dot{H}^s$-norm of the solution.}

Before giving the next lemma, let denote the different combinations of derivatives useful for the following. Let $\bi$ a multi-index of length $s$. In the next computations, we will deal with $(p+1)$-linear forms, with different combinations of derivatives:
\begin{align*}
L(f_0, \cdots, f_p)= \int \partial_\bi f_0 \partial_{1 \bi} (f_1 \cdots f_p). 
\end{align*}

We need now to distribute the derivatives of $\partial_{1 \bi}(f_1\cdots f_p)$, in the following way: the function $f_1$ receives the set $\bj_1$ of derivatives among $i_1, \cdots, i_s,1$, \dots, the function $f_p$ which receives the set $\bj_p$ of derivatives among the same set. One particular case, which we want to consider separately, is when all the derivatives fall on only one function: for this, we introduce the set $I_\bi$ by:
\begin{align*}
I_\bi := \left\{ (\bj_1, \cdots, \bj_p)\in \left( \mathcal{P} ( i_1, \cdots,i_s, 1 ) \right) ^p ; \bj_1 \cup \cdots \cup \bj_p=\bi \cup \{1\} =(i_1, \cdots, i_s, 1), \text{ with at least two } \bj_k \text{ not empty} \right\},
\end{align*}
and a special subset:
\begin{align*}
\widetilde{I_\bi}:= \left\{ (\bj_1, \cdots, \bj_p) \in I_\bi ; \text{ with at least one } \bj_k \text{ satisfies } \vert \bj_k \vert=s \right\}.
\end{align*}
In other words, the set $\widetilde{I_\bi}$ characterized the terms with all the derivatives on one term, except one on another term. Finally, we want to find out the term with those $s$ derivatives:
\begin{align*}
\widetilde{I_\bi}(f_k)= \left\{ (\bj_1, \cdots, \bj_p) \in \widetilde{I_\bi} ; \vert \bj_k \vert=s \right\}.
\end{align*}

The following lemma establishes a bound on the time evolution of $s$-derivatives of the solution, up to some terms with $s$ derivatives on each error $\v$. Those terms correspond to those in (\ref{equation10}), which we will deal with Lemma \ref{lem:mono2}.

\begin{lemm}
Let $\bi$ be a multi-index of length $s$. The following pointwise estimate holds, at each time $t$:
\begin{align}\label{estimeederivee}
\MoveEqLeft
\left\vert \frac{d}{dt} \int u_\bi - p(p-1) \int \v_\bi^2 R^{p-2} \partial_1 R+ 2 \sum_{(j_1,\cdots,j_p) \in \widetilde{I_\bi}(\v)} \int \v_\bi \v_{j_k} \prod_{k'\neq k} R_{j_{k'}} \right\vert \lesssim G_s(t).
\end{align}
\end{lemm}

\begin{proof}
Let us compute the time derivative, with $u$ a solution of (\ref{ZK}):
\begin{align*}
\frac{d}{dt} \int u_\bi^2 = -2 \int u_\bi \partial_{1 \bi} (u^p).
\end{align*}

Let us now distribute the $s+1$ derivatives. Two cases occur:
\begin{itemize}
\item either the $s+1$ are all on the same term, and we obtain:
\begin{align*}
-2p \int u_\bi (\partial_1 u_\bi) u^{p-1} = p (p-1) \int (u_\bi)^2 u^{p-2}\partial_1 u.
\end{align*}
\item or the derivatives are not all on the same term, we find back the definition of $I_\bi$.
\end{itemize}

We thus develop the time derivative into:
\begin{align}\label{petit_dvt}
\frac{d}{dt} \left( \int u_\bi^2 \right) = p(p-1) \int (u_\bi)^2 u^{p-2} \partial_1 u - 2 \sum_{(\bj_1, \cdots,\bj_p) \in I_\bi} \int u_\bi \prod_{k=1}^p u_{\bj_k}.
\end{align}

To obtain the estimate (\ref{estimeederivee}), we replace in (\ref{petit_dvt}) $u$ by $\v+R$, we develop, and then estimate each term. By developing, the $(p+1)$-linear forms are applied $l$ times on $\v$ and $p+1-l$ times on $R$. We decompose the different cases.

\vspace{0.25cm}\textbf{First case : $l=0$.} We obtain the terms:
\begin{align*}
p(p-1) \int (R_\bi)^2 R^{p-2} \partial_1 R - 2 \sum_{(\bj_1, \cdots,\bj_p) \in I_\bi} \int R_\bi \prod_{k=1}^p R_{\bj_k}.
\end{align*}
The number of derivatives is odd, then by parity of each soliton, this $(p+1)$-linear form applied to the same soliton $(R^k, \cdots, R^k)$ is null. Only remains the interaction between the different solitons. We thus get by (\ref{interaction}):
\begin{align*}
\left\vert p(p-1) \int (R_\bi(t))^2 R^{p-2}(t) \partial_1 R(t) - 2 \sum_{(\bj_1, \cdots,\bj_p) \in I_\bi} \int R_\bi(t) \prod_{k=1}^p R_{\bj_k}(t) \right\vert \leq Ce^{-\frac{1}{2}\sigma_0 \sqrt{\sigma_0}t}.
\end{align*}

\vspace{0.25cm}\textbf{Second case : $l=1$.} In that case, it means that $\v$ appears only once in the expression (\ref{petit_dvt}), with a certain number of derivatives. By integration by parts, all the derivatives can be put on the other terms, and those terms are finally bounded by a Cauchy-Schwarz inequality by $C\| \v(t) \|_{L^2}$.

\vspace{0.25cm}\textbf{Third case : $l=2$.} We need to identify how many derivatives the terms $\v$ possesses. In both $\v$ have exactly $s$ derivatives, we find the terms in (\ref{estimeederivee}) bilinear in $\v$. Observe that the sum is made on $\widetilde{I_\bi}$. Those terms will need the monotonicity argument, developed in the previous subsection, and applied in the next.

The second possibility is that the two errors $\v$ do not possess $s$ derivatives. By a Cauchy-Schwarz inequality, it implies that there are bounded by $C \| \v \|_{H^s} \| \v \|_{H^{s-1}}$, or more generally by:
\begin{align*}
C \| \v \|_{\dot{H}^s}^{2-\frac{1}{s}} \| \v \|_{H^1}^\frac{1}{s} + C \| \v \|_{H^1}^2.
\end{align*} 

\vspace{0.25cm}\textbf{Fourth case : $l=3$ or $l=4$.} First for $l=3$, let focus on the number of derivatives on the three $\v$. In the worst case, one has $s$ derivatives, and the two others have $s$ and $1$, or $s-1$ and $2$ derivatives. We obtain, with $\vert \bj_1 \vert=s$, and $\vert \bj_2\vert =1$, and by the embedding $H^2(\mathbb{R}^d) \hookrightarrow L^\infty(\mathbb{R}^d)$:
\begin{align*}
\left\vert \int \v_\bi \v_{\bj_1} \v_{\bj_2} R^{p-2} \right\vert 
	& \leq C \| \v \|_{\dot{H}^s}^2 \| \v_{\bj_1} \|_{L^\infty} \leq C \| \v \|_{\dot{H}^s}^2 \| \v \|_{H^3}.
\end{align*}

If we consider $\vert \bj_1 \vert=s-1$, and $\vert \bj_2\vert =2$:
\begin{align*}
\left\vert \int \v_\bi \v_{\bj_1} \v_{\bj_2} R^{p-2} \right\vert 
	& \leq C \| \v \|_{\dot{H}^s} \| \v \|_{\dot{H}^{s-1}} \| \v_{\bj_2} \|_{L^\infty} \leq C \| \v \|_{\dot{H}^s} \| \v \|_{\dot{H}^{s-1}} \| \v \|_{H^4} \\
	& C \leq \| \v \|_{\dot{H}^s} \left( \| \v \|_{\dot{H}^s} \| v\|_{H^3} + \| \v \|_{H^1}^2 \right).
\end{align*}

The other terms with $l=3$ have a lower number of derivatives, and thus are easier to deal with. 

Now for $l=4$, this situation is possible if $p=3$, and the bound is similar to $l=3$ with $\dot{H}^2(\mathbb{R}^d) \hookrightarrow L^\infty(\mathbb{R}^d)$. 

Since $p\leq 3$, there is no other case for $l$.
\end{proof}

\begin{rem}
Notice that in the case $l=3$, with $\bj_1=s-1$ and $\bj_2=2$, the bound on $s\geq s_0=4$ is necessary in order that the argument works.
\end{rem}

\subsection{$H^4$ and $H^s$ bounds of the errors}

Now we establish a bound of the $\dot{H}^s$-norm for $s\geq 4$. Recall that the error $\v(t)=(u-R)(t)$ (equal to $0$ at time $S_n$), has a nice $H^1$-norm on $[T_0,S_n]$ with a bound given in the previous section:
\begin{align} \label{borneH1}
\forall t \in [T_0,S_n], \quad \| \v(t) \|_{H^1} \leq A_1 e^{-\delta_1 t}.
\end{align}
Let us continue the discussion of the beginning of this section, before a rigorous explanation. By the previous computation, we obtain that, for time $t$ when the solution $u$ exists:
\begin{align*}
\left\vert \| v(t) \|_{\dot{H}^s}^2 - \|v(S_n) \|_{\dot{H}^s}^2\right\vert \leq C_s \left( \|v(t) \|_{\dot{H}^{s-1}}^2 + \int_t^{S_n} G_s(t') dt' + O(l.o.t) \right).
\end{align*}

$G_s(t')$ contains two main terms which lead the exponential decay.

First, $G_s(t')$ contains bilinear terms in $v$ of the form $\|v(t) \|_{\dot{H}^s}^{2-\frac{1}{s}} \|v(t)\|_{H^1}^\frac{1}{s}$. If we consider only the bilinear terms, the adequate power closes the bootstrap argument, and we obtain a bound $\| v(t) \|_{H^s}^2 \leq A_s e^{-\delta_s t}$, on the time interval $[T_0,S_n]$, with a constant $A_s$ growing exponentially with $s$.

Second, $G_s(t')$ contains trilinear terms in $v$, on the form $\|v(t)\|_{\dot{H}^s}^2 \|v(t) \|_{H^3}$. If we consider the bound as $\| v(t) \|_{H^s}^3$, the previous computation with the bilinear terms still holds, but on a time interval $[T_s,S_n]$, with no way to complete before $T_s$ to a time independant of $s$ : a bootstrap type argument gives the bound $\| v(t) \|_{H^s}^2 \leq A_s e^{-\delta_s t}$, on a time interval $[T_s,S_n]$. The trilinear terms are the source of the dependance of $T$ on $s$. We require a uniform time $T$ on which we obtain the exponential decay of all $H^s$-norms. The procedure needs to be accurate : we first deal with the $H^4$-norm, and obtain the adequate decay on $[T_4, S_n]$. We thus bound the trilinear term by $\|v(t)\|_{\dot{H}^s}^2 A_4 e^{-\delta_4 t}$. We define for each $s$ larger than $5$, a time $T_s$ on which the bound of the trilinear terms is lower than the bound of the bilinear terms. We use a bootstrap type argument on $[T_s,S_n]$ with an estimate on the time $T_s \geq T_4$. On the interval $[T_4,T_s]$, we bound the trilinear term by $\|v(t)\|_{\dot{H}^s}^2 A_4 e^{-\delta_4 T_4}$, and then complete by a bootstrap type argument and compactness. This procedure influences the constant $A_s$, but the exponential decay with the same coefficient $\delta_4$ holds on $[T_4,S_n]$.

\medskip

Let us now detail the computation to prove the proposition \ref{propo_decroit_Hs}. 

\subsubsection*{$H^4$-norm}

Let $s= 4$. We prove the following bootstrap. There exist constants $A_4$, $\delta_4:=\frac{\delta_1}{2}$ and a time $T_4\geq T_0$ such that the following bootstrap holds. Let $t^* \in [T_4,S_n]$. If
\begin{align}\label{bootstrap_assumption_H4}
\forall t \in [t^*,S_n], \quad \| v(t) \|_{\dot{H}^4} \leq A_4 e^{-\delta_4 t},
\end{align}
then 
\begin{align}\label{bootstrap_conclusion_H4}
\forall t \in [t^*,S_n], \quad \| v(t) \|_{\dot{H}^4} \leq \frac{1}{2} A_4 e^{-\delta_4 t}.
\end{align}

If this bootstrap is true, because $\|v(S_n) \|_{\dot{H}^4}=0$, it immediately proves the proposition \ref{propo_decroit_Hs} for $s=4$.

Let assume the bootstrap assumption (\ref{bootstrap_assumption_H4}) with constants $A_4$, $T_4$ to define later. With $\vert \bi \vert=4$, the estimate (\ref{estimeederivee}) with the help of (\ref{equation10}) can now be integrated into:
\begin{align*}
\MoveEqLeft
\left\vert \int u_\bi^2(S_n)-\int u_\bi^2(t) \right\vert \leq C_4 \left( \|v(t)\|_{\dot{H}^3}^2 + \int_t^{S_n} G_4(t') dt' \right).
\end{align*}

We obtain a bound on the error:
\begin{align}
\MoveEqLeft
\sum_{\vert \bi \vert =4} \int \v_\bi^2(t) = - \sum_{\vert \bi \vert =4} \int_{t}^{S_n} \frac{d}{dt}\left( \int (u-R)_\bi^2(t') \right) dt' \nonumber\\
	& = \sum_{\vert \bi \vert =4} \left( \int u_\bi^2(t) -\int u_\bi^2(S_n) \right) +2(-1)^{1+4}\sum_{\vert \bi \vert=4} \left( \int u(t)R_{\bi\bi}(t) -\int u(S_n)R_{\bi\bi}(S_n) \right) \nonumber\\
		& \quad + \sum_{\vert \bi \vert=4} \left( \int R_\bi (t) R_\bi(t) - \int R_\bi (S_n) R_\bi(S_n) \right) \nonumber\\
	& \lesssim \sum_{\vert \bi \vert =4} \left\vert \int u_\bi^2(t) -u_\bi(S_n) \right\vert + \| u(t)-R(t) \|_{L^2} + \| u(S_n)-R(S_n) \|_{L^2} + e^{- \frac{\sigma_0}{2} \sqrt{\sigma_0}t} \nonumber\\
	& \lesssim \| v(t) \|_{\dot{H}^3}^2 + \int_t^{S_n} G_4(t') dt' + \| v(t) \|_{L^2} + e^{- \frac{\sigma_0}{2} \sqrt{\sigma_0}t} \label{computations}\\
	& \leq C_{\sigma_0,A_1, \delta_1, \delta_4} \left( e^{- \delta_1 t} +  \left( A_4 e^{-\delta_4 t} \right)^{2-\frac{1}{4}}\left( A_1 e^{-\delta_1 t} \right)^{\frac{1}{4}}  + \left( A_4 e^{-\delta_4 t}\right)^3 \left( 1+ A_1 e^{-\delta_1 t} \right)^{p-2} \right). \nonumber
\end{align}
This estimate is sufficient to conclude. In fact, by taking $\delta_4:=\frac{\delta_1}{2}$, we choose $A_4$ and $T_4$ large enough such that:
\begin{align*}\MoveEqLeft
C_{\sigma_0,A_1, \delta_1} \left( e^{- \delta_1 t} +  \left( A_4 e^{-\delta_4 t} \right)^{2-\frac{1}{4}}\left( A_1 e^{-\delta_1 t} \right)^{\frac{1}{4}}  +  \left( A_4 e^{-\delta_4 t}\right)^3 \left( 1+ A_1 e^{-\delta_1 t} \right)^{p-2} \right) \leq \frac{A_4^2}{2^2} e^{- 2\delta_4 T_4}.
\end{align*}
This gives the bootstrap conclusion (\ref{bootstrap_conclusion_H4}).

We now fixed $T_4$, and $\delta_4 =\frac{\delta_1}{2}$.

\subsubsection*{$H^s$-norm for $s \geq 5$}

As pointed out before, the arguments for $s=4$ almost hold to prove (\ref{borneHs}) for $s\geq 5$. However, the bound holds until a time $T_s$, and the exponential decay is not uniform on a time interval. We thus separate the interval into two pieces, and prove the following bootstrap. There exists a constant $B_s$, such that the following holds. Let the time $T_s := \max\left(T_4, \frac{1}{\delta_4} \ln (A_4 B_s^\frac{1}{s})\right)$, and $t^*\in [T_s,S_n]$. If 
\begin{align}\label{bootstrap_assumption_Hs}
\forall t \in [t^*,S_n], \quad & \| v(t)\|_{\dot{H}^s} \leq B_s A_4 e^{- \delta_4 t},
\end{align}
then 
\begin{align}\label{bootstrap_conclusion_Hs}
\forall t \in [t^*,S_n], \quad & \| v(t)\|_{\dot{H}^s} \leq \frac{1}{2} B_s A_4 e^{- \delta_4 t}.
\end{align}

\medskip

Let us suppose that this bootstrap holds on $[T_s,S_n]$. We obtain the exponential decay on $[T_s, S_n]$. A bootstrap type argument with $\|v(t)\|_{\dot{H}^s}^2 \|v(t) \|_{H^3} \leq \|v(t)\|_{\dot{H}^s}^2 A_4 e^{-\delta_4 T_4}$, shows that $\| v(t) \|_{H^s}^2$ is bounded on $[T_4,T_s]$ by a certain constant $\widetilde{B}_s$. Gathering those tow results, we obtain the exponential decay of $\|v(t)\|_{H^s}$ on $[T_4, S_n]$, and concludes the proof of proposition \ref{propo_decroit_Hs}.

\medskip

It remains, assuming (\ref{bootstrap_assumption_Hs}), to deduce (\ref{bootstrap_conclusion_Hs}). The proof focuses on the case $p=2$, the same holds for $p=3$. By the same computations made in (\ref{computations}), the expression (\ref{G_s}) of $G_s$, the bound (\ref{bootstrap_assumption_Hs}) we obtain:
\begin{align}
\MoveEqLeft
\sum_{\vert \bi \vert=s} \int v_\bi^2(t) \leq C_s \left(e^{-\frac{\sigma_0}{2}\sqrt{\sigma_0}t} + \|v(t) \|_{L^2} + \| v(t) \|_{\dot{H}^{s-1}}^2 + \int_t^{S_n} G_s(t') dt' \right) \nonumber\\
	& \leq C_s \left(e^{-\frac{\sigma_0}{2}\sqrt{\sigma_0}t} + \|v(t) \|_{L^2} + \| v(t) \|_{\dot{H}^s}^{2-\frac{2}{s}} \| v(t) \|_{\dot{H}^4}^\frac{2}{s} + \int_t^{S_n} e^{-\delta_1 t'} + \| v(t) \|_{\dot{H}^s}^{2-\frac{1}{s}} \| v(t') \|_{H^1}^{\frac{1}{s}} + \|v(t') \|_{\dot{H}^s}^2 \|v(t')\|_{H^4} dt' \right). \label{certains_calculs}
\end{align}
By the special choice of $T_s$, the bound on the last trilinear term in the integral is bounded by the bound of the bilinear term:
\begin{align*}
\left( B_s A_4 e^{-\delta t} \right)^2 A_4 e^{-\delta_4 t} \leq \left( B_s A_4 e^{-\delta t} \right)^{2-\frac{1}{s}} (A_4 e^{-\delta_4 t})^\frac{1}{s}
	& \Leftrightarrow  \left( B_s \right)^\frac{1}{s} A_4 e^{-\delta_4 t} \leq 1 \\
	& \Leftrightarrow t \geq T_s \geq \frac{1}{\delta_4}\ln(A_4 B_s^\frac{1}{s}).
\end{align*}
We thus obtain:
\begin{align*}
\sum_{\vert \bi \vert=s} \int v_\bi^2(t) 
	& \leq C_s \left( e^{-\frac{\sigma_0}{2}\sqrt{\sigma_0} t}+ A_1e^{-\delta_1 t} + B_s^{2-\frac{2}{s}} A_4^2 e^{-2\delta_4 t} + \int_t^{S_n} e^{-\delta_1 t'} + B_s^{2-\frac{1}{s}} A_4^2e^{-2 \delta_4 t'} dt' \right) \\
	& \leq C_s \left( \left( 1+ \frac{2}{\sigma_0 \sqrt{\sigma_0}} \right) e^{-\frac{\sigma_0}{2}\sqrt{\sigma_0} t} + \left( 1 + \frac{1}{\delta_1} \right) e^{-\delta_1 t} + \left( 1+ \frac{1}{2\delta_4}\right) B_s^{2- \frac{1}{s}} A_4^2 e^{-2 \delta_4 t} \right).
\end{align*}
By choosing any $B_s$ satisfying:
\begin{align*}
C_s \left( \frac{1}{A_4^2} \left( 1 + \frac{2}{\sigma_0 \sqrt{\sigma_0}}+ \frac{1}{\delta_1}  \right)+ \left( 1 + \frac{1}{2\delta_4} \right)B_s^{2-\frac{1}{s}}\right) \leq \frac{1}{2}B_s^2,
\end{align*}
the conclusion of the bootstrap estimate (\ref{bootstrap_conclusion_Hs}) is then proved.

\begin{rem} A dependency to underline is that $B_s \geq \left( \frac{2C_s}{\delta_4} \right)^s$, with the constant $C_s$ depending on the number of combination of $s$ derivatives: $C_s \sim d^s$. In fact, the optimal constant is expected to be $K^s$, with $K$ large enough : it is not reach here.
\end{rem}

\section{End of the construction of a smooth multi-soliton}\label{constructionfinale}

We now complete the existence part of Theorem \ref{th1}. At this point, we dispose of a sequence of solutions $(u_n)_n$ of \eqref{ZK} defined on $[T_0,S_n]$ for some fixed $T_0 \in \mathbb R$ and $S_n \to +\infty$, and such that for all $s \ge 0$, there exist $A_s>0$ such that for all $n \in \mathbb N$
\begin{equation}\label{est:Hs_un_R}
\forall t \in [T_0,S_n], \quad \| u_n(t) - R(t) \|_{H^s} \le A_s e^{- \delta_1t/2}.
\end{equation}

Consider the sequence $(u_n(T_0))_n$: it is bounded in each $H^s$, and so, up to a subsequence that we still denote $(u_n(T_0))_n$, it admits a weak limit $U_0$ which belongs to all $H^s$ for $s \ge 0$. 

Let us show that the convergence is actually strong.

\begin{lemm}
$(u_n(T_0))_n \to U_0$ strongly in $L^2$, and more generally, $H^s$, for all $s\geq 0$.
\end{lemm}

\begin{proof}
It suffices to prove:
\begin{align}\label{astuce}
\forall \epsilon >0, \exists K_\epsilon>0, \forall n, \quad \int_{\vert x \vert >K_\epsilon} \vert u_n(T_0) \vert^2 <\epsilon.
\end{align}
The bound (\ref{astuce}) comes from the estimate (\ref{hypo}). In fact, let $t^*>T_0$ such that $A_1 e^{-\frac{1}{L}\frac{\sigma_0}{8}t^*} \leq \frac{\epsilon}{2}$. The aim is to control the evolution of the mass outside a compact set. Let a function $g\in \mathcal{C}^3(\mathbb{R}^+)$, null on $[0,1]$ and equal to $1$ on $[2, \infty]$. With $\gamma >0$ and $K_\epsilon>0$ to define later, the evolution of the mass outside a compact set gives:
\begin{align*}
\left\vert \frac{d}{dt} \int u_n^2(t,\bx) g \left( \frac{\vert \bx \vert -K_\epsilon}{\gamma} \right) \right\vert 
	& = \left\vert - \int d_1 g \vert \nabla u_n \vert^2 -2 \int d_1 g \left( \partial_1 u_n \right)^2 + \int d_1^3 g u_n^2 + \frac{4}{3} \int d_1 g u_n^3 \right\vert \\
	& \leq \frac{C}{\gamma}\| u_n \|_{H^1}^2 \left( 1 + \| u_n \|_{H^2} \right),
\end{align*} 
where the last term comes from Sobolev embedding. By (\ref{hypo}), the $H^1$-norm is uniformly bounded by a constant $C$, so it suffices to take $K_\epsilon$ and $\gamma$ large enough to obtain:
\begin{align*}
\left\vert \frac{d}{dt} \int u_n^2(t,x) g \left( \frac{\vert \bx \vert -K_\epsilon}{\gamma} \right) \right\vert \leq \frac{\epsilon}{2(t^*-T_0)}.
\end{align*}
Choose $\gamma$ larger to obtain $\| R(t^*)\mathbf{1}_{\vert x \vert >K + \gamma} \|_{H^1}^2 \leq \frac{\epsilon}{4}$, and integrate from $T_0$ to $t^*$ of the previous equality imply:
\begin{align*}
\int_{\vert \bx \vert \geq 2(K_\epsilon+\gamma)} u_n(T_0)^2  \leq \frac{\epsilon}{2}+ \int_{\vert \bx \vert \geq K_\epsilon+\gamma} u_n(t^*)^2 \leq \frac{\epsilon}{2}+2 \int_{\vert \bx \vert \geq K_\epsilon+\gamma} R(t^*)^2  + 2 \|(u_n-R)(t^*) \|_{H^1}^2 \leq \epsilon,
\end{align*}
and concludes the proof of (\ref{astuce}), and the strong convergence of $\left(u_{n}(T_0)\right)_k \to U_0$ in $L^2$. Given $s \ge 0$, by interpolation with $H^{s+1}$ (where weak convergence hold), we conclude that $\left(u_{n}(T_0)\right)_n \to U_0$ strongly in $H^s$.
\end{proof}

Now consider the solution $R^*$ of (\ref{ZK}) with initial data $R^*(T_0):=U_0$, defined on the maximal interval to the right $[T_0,T_+)$.

Let $s > 1/4$ and $t \in [T_0,T_+)$. As $u_n(T_0) \to R^*(T_0)$ in $H^s$, due to the continuity of the flow in $H^s$ (see Theorem \ref{th:lwp} in Appendix B, where the local well posedness theory is recalled), we obtain that $u_n(t) \to R^*(t)$ in $H^s$. In particular, taking the limit in $n$ in \eqref{est:Hs_un_R} we obtain
\begin{equation} \label{est:Hs_R*_R}
\| R^*(t) - R(t) \|_{H^s} \le A_s e^{- \delta_1t/2}.
\end{equation}
By inspection, $\| R(t) \|_{H^s}$ is bounded for $t \in \mathbb R$, so that $\| R^*(t) \|_{H^s}$ is bounded on $[T_0,T_+)$. Due to the blow up criterion in  Theorem \ref{th:lwp}, we infer that $T_+ = +\infty$ (this part of the arugment is only relevant for (\gls{ZK23})). The bound \eqref{est:Hs_R*_R} is therefore valid for all $t \ge T_0$: this means that $R^*$ is the desired multi-soliton associated to $R$.

Finally, observe that for all $t \in T_0$, $R^*$ is smooth (it lies in all $H^s$, $s \ge 0$), so that, using the equation \eqref{ZK}, we see that $\partial_t R^*$ is smooth too, and by a straightforward induction, $R^* \in \mathcal C^\infty([T_0, +\infty) \times \mathbb R^d)$.

This concludes the existence part of Theorem \ref{th1}.

\bigskip

Now we establish a corollary on any multi-soliton. If $u$ is a multi-soliton in the sense of the definition (\ref{defi_multi}), then the convergence of $u$ to the sum of the $K$ decoupled soliton is exponential.

\begin{coro}\label{coro_multisol}
Let $u \in \mathcal C([T,+\infty),H^1(\mathbb{R}^d)$ be a multi-soliton solution of (\ref{ZK}) in the sense of definition (\ref{defi_multi}), and denote $R$ the associated profile. Then the convergence of $u$ to $R$ occurs at at an exponential rate: there exist $A_1>0$, and $\gamma_1>0$ such that:
\begin{align*}
\forall t \geq T, \quad \| u(t)-R(t) \|_{H^1} \leq A_1e^{-\gamma_1 t}.
\end{align*}
\end{coro}

\begin{proof}
It follows the ideas of Martel \cite[Proposition 4]{Mar}.
Proving this lemma is equivalent to proving Proposition \ref{propo5}, except that we consider $u$ instead of $u_n$. Let us consider a sequence of time $S_n\rightarrow \infty$, and the solutions with the initial conditions $u_n(S_n)=u(S_n)$ on the interval $[T,S_n]$ (by uniqueness, for $t \in [T, S_n]$, $u(t)=u_n(t)$). By assumption, the sequence $\| u_n(S_n)-R(S_n) \|_{H^1}$ goes to $0$. If we can prove the proposition \ref{propo5} for this new $(u_n)_n$, it concludes the proof of the corollary.

To prove (\ref{decay}) for the new $u_n$, we mimic the proof that we already done. The difference is then only in the initial condition. The parameters $L$, $\alpha$, $T_1$ and $A_1$ are to be found during the proof. The lemma of modulation applies similarly. The lemmas on the evolution of the masses and energies are identical : in fact, they concern the time derivative of these quantities, so it does not see the change of initial condition. The formula (\ref{global}) see in fact the change of initial conditions $E^k(S_n)$ and $M^k(S_n)$. However, the arguments of monotonicity apply similarly, and allow to get rid of those terms. The argument of coercivity is then exactly the same.
\end{proof}

\begin{rem}
At this point, nothing is known about the convergence in $H^s$ of $u(t) - R(t)$, since we ignore if the initial condition $u(t)$ belongs $H^s$ for $s >1$. In fact, by uniqueness, this will be the case!
\end{rem}

\section{Uniqueness}\label{unicite}

The goal of this section is to prove uniqueness in the following sense. If two solutions behave both as $t\rightarrow +\infty$ to the same multisoliton, then they are equal. For now, we denote $R^*(t)$ the solution established in the previous section, on a time interval $[T_0, \infty)$. This is a multisoliton close to $R(t)$.

\begin{prop}
Let $u \in \mathcal{C}([T_0, \infty),H^1)$ be a solution of (\ref{ZK}), and satisfying:
\begin{align*}
\| u(t) - R(t) \|_{H^1} \underset{t\rightarrow \infty}{\rightarrow} 0.
\end{align*}
Then $u \equiv R^*$.
\end{prop}

The proof is inspired by the techniques used by Martel on the subcritical and critical (gKdV) equations in \cite{Mar}, Proposition 6.

\begin{proof}
First, we can notice that it is equivalent to consider that $u-R^*$ or $u-R$ goes to $0$ at infinity. For now, we will use the following notations:
\begin{align*}
u(t)= z(t)+R^*(t).
\end{align*}
For now, we prove the uniqueness for (\gls{ZK22}) and (\gls{ZK32}). The critical case (\gls{ZK23}) is dealt with at the end.

The new function $z$ satisfies the equation:
\begin{align*}
z_t + \partial_1 \left( \Delta z + z^2 + 2zR^* \right)=0.
\end{align*}

To follow the scheme on the existence, we decompose the error along the different directions, which are the derivatives along an $l$-axis of the $k^{\text{th}}$ soliton:
\begin{align*}
\widetilde{z}(t):= z(t) - \sum_{i=1}^d \sum_{k=1}^K a^{ik} R^k_i, \text{ with } a^{ik} (t) := \frac{1}{\int \left(R^k_i(t) \right)^2dx} \int R^k_i(t) z(t) dx.
\end{align*}
We obtain the following estimates, for some constants $C_1$ and $C_2$:
\begin{align*}
\forall t>T_0, \quad C_1 \| z (t) \|_{H^1} \leq \| \widetilde{z}(t)\|_{H^1} + \sum_{i,k} \vert a^{ik}(t) \vert \leq C_2 \| z(t) \|_{H^1}.
\end{align*}

We need to show that:
\begin{align}\label{principal}
\| \widetilde{z}(t) \|_{H^1} + \sum_{ik} \vert a^{ik}(t) \vert \leq C e^{-\gamma_1 t} \sup_{t' \geq t} \| z(t') \|_{H^1}.
\end{align}

If we admit this inequality, we conclude that for $t$ large enough, $\|z(t) \|_{H^1}=0$, and so $u\equiv R^*$. Let show this inequality in different steps.

\vspace{0.5cm}\textbf{Step 1.} Estimate on $\widetilde{z}(t)$.

We use the decreasing function $\htilde(t,\bx):= \sum\limits_{k=1}^K \frac{1}{c^k} \phit ^k (\bx)$. Let recall that $m^k$ are defined in (\ref{mk}). $\htilde$ is close to $\frac{1}{c^k}$ where the $j^{\text{th}}$ soliton is localised:
\begin{align*}
\left\| \left( \htilde(t,\bx) -\frac{1}{c_1} \right)\mathbf{1}_{]-\infty, \frac{1}{2}\left(m^1(t) +c^1t +y^1_1 \right)]} \right\|_{L^\infty} \leq C e^{-\gamma_1t}, \\
\forall 2 \leq k \leq K \text{ , } \left\| \left( \htilde(t,\bx) -\frac{1}{c^k} \right)\mathbf{1}_{[\frac{1}{2}\left(m^{k-1}(t) +c^kt +y^k_1\right), \frac{1}{2}\left(m^k(t) +c^kt +y^k_1 \right)]} \right\|_{L^\infty} \leq C e^{-\gamma_1t}, \\
\left\| \left( \htilde(t,\bx) -\frac{1}{c^K} \right)\mathbf{1}_{[\frac{1}{2}\left(m^{K-1}(t) +c^Kt +y^K_1 \right), + \infty[} \right\|_{L^\infty} \leq C e^{-\gamma_1t}.
\end{align*} 
In the other regions where the variations of $h$ are higher, the solitons collapse, and we will not see those variations. We claim the estimate, for a constant $\lambda_2>0$:
\begin{align}\label{coercivite}
 \int \left( \vert \nabla \widetilde{z} \vert^2 -2R \widetilde{z} ^2 \right)\htilde + \widetilde{z} ^2 \geq \lambda_2 \| \widetilde{z} \|_{H^1} ^2- \frac{1}{\lambda_2 } \sum_k \left( \left\vert \int \widetilde{z} R^k \right\vert^2 + \sum_i \left\vert \int \widetilde{z} R^k_i \right\vert^2 \right).
\end{align}
The proof of this estimate of coercivity is close to the one obtained for the existence, so we will skip it.

We already know that $\int \widetilde{z} R_i^k =0$ due to the orthogonality. We have  to handle the other term. In fact, because of the weak interaction between solitons:
\begin{align*}
\left\vert \sum_{i,k} a^{ik}(t) \int R_i^j R^k \right\vert \leq C e^{-\gamma_1 t} \sup_{t'>t} \| z(t') \|_{H^1},
\end{align*}
we obtain $\left\vert \int \widetilde{z} R^k -\int z R^k \right\vert \leq Ce^{-\gamma_1 t} \sup_{t'>t} \| z(t') \|_{H^1}$ which allows to focus on $\int z R^k$.

\vspace{0.5cm}\textbf{Step 1.1} Control of $\int z R^k$. By the equation satisfied by a soliton, we have $0=-c^k R_1^k + \Delta R_1^k + 2 R^kR^k_1$, and it implies:
\begin{align*}
\left\vert \frac{d}{dt} \int R^k z \right\vert 
	& = \left\vert \int R_1^k \left( -c^k z + \Delta z + (z+R^*)^2-(R^*)^2 \right)\right\vert \\
	& \leq \left\vert \int R_1^k \left( (z+R^*)^2-{R^*}^2-2R^*z \right) \right\vert + \left\vert \int R_1^k \left( 2R^*-2R \right) z \right\vert + \left\vert \int R_1^k \left( 2R-2R^k \right) z \right\vert  \\
		& \quad+ \left\vert \int( 2 R^k R^k_1 -c^k R^k + \Delta R^k ) z \right\vert \\
	& \leq Ce^{-\gamma_1 t} \| z (t) \|_{L^2}.
\end{align*}

In the previous estimate, we notice that we earn an error of $z$ because we made the difference between two solutions, instead of the difference of a solution with $R$ : see the comparison with (\ref{une_direction}). Then integrating the previous estimate from $t$ to $+\infty$:
\begin{align*}
\left\vert \int R^k z \right\vert \leq Ce^{-\gamma_1t} \sup_{t'>t} \| z(t') \|_{H^1}.
\end{align*}
We can now modify (\ref{coercivite}) into:
\begin{align} \label{coercivite2}
\lambda_2 \| \widetilde{z}(t) \|_{H^1}^2 \leq Ce^{-\gamma_1 t} \sup_{t'>t} \| z(t') \|_{H^1}^2 + \int \left( \vert \nabla \widetilde{z} (t) \vert ^2 -2 R(t) \widetilde{z}^2(t)  \right) h(t) + \widetilde{z}(t)^2.
\end{align}

\vspace{0.5cm}\textbf{Step 1.2} Control of the operator $L_h$.
Let the operator $L_h$ defined by $L_h(z):= (-\Delta z-2Rz)h+z$. An immediate computation gives us:
\begin{align*}
\MoveEqLeft
\int \left( \vert \nabla \widetilde{z} \vert ^2 - 2 R \widetilde{z} ^2 \right)h + \widetilde{z} \\
	& = \int \left( \vert \nabla z \vert ^2 -2 R z \right)h +z^2 - \frac{1}{2} \int z^2 \Delta h + \frac{1}{2} \int \widetilde{z}^2 \Delta h - \sum_{ik} a^{ik} \int L_h z R^k_i - \sum_{ik} a^{ik} \int L_hR_i^k z \\
		& \quad + \sum_{i_1, i_2, k_1, k_2} a^{i_1 k_1} a^{i_2 k_2} \int L_h R_{i_1}^{k_1} R_{i_2}^{k_2}.
\end{align*}
The first term will be dealt with in the next step. We need to watch how the operator $L_h$ acts on the $i^{\text{th}}$-derivative of the $k^{\text{th}}$ soliton.
\begin{align*}
L_hR^j_i= \left( -\Delta R_i^k -2 R R_i^k \right) \left( h-\frac{1}{c^k} \right) + \frac{1}{c^k} \left( -\Delta R_i^k -2RR_i^k +c^kR_i^k \right)
\end{align*}
implies the following control:
\begin{align*}
\vert L_h R_i^k \vert \leq Ce^{-\gamma_1 t} e^{-\frac{\sigma_0}{2} \vert \bx-\by^k -c^kt\beone \vert}.
\end{align*}
Furthermore, we can control the term which enables the operator $L_h$ to be self-adjoint:
\begin{align*}
\int L_h z R_i^k = \int z L_h R_i^k - \int z \left( \Delta h R_i^k + 2 \nabla h \cdot \nabla R_i^k \right),
\end{align*}
which gives, by the localization of the derivatives of $h$:
\begin{align*}
\left\vert \int L_h z R_i^k \right\vert + \left\vert \int L_h R_i^k z \right\vert \leq Ce^{-\gamma_1 t} \| z(t) \|_{L^2}.
\end{align*}

By a Cauchy-Schwarz inequality on the term $a^{ij}$, we obtain:
\begin{align}\label{coercivite3}
\int \left( \vert \nabla \widetilde{z} \vert^2 -2R \widetilde{z} ^2 \right)h + \widetilde{z}^2 \leq Ce^{-\gamma_1 t} \| z \|_{L^2}^2 +\int \left( \vert \nabla z \vert^2 -2Rz^2 \right)h +z^2.
\end{align}

\vspace{0.5cm}\textbf{Step 1.3} Control of $\int \left( \vert \nabla z \vert^2-2Rz^2 \right)h+z^2$.
In order to study the previous quantity, we express the operator as a linearisation by the function:
\begin{align*}
F(t,z):= 2 \left( \frac{(z(t)+R^*(t))^3}{3} -\frac{R^*(t)^3}{3}-{R^*}^2(t)z(t) \right).
\end{align*}
We obtain, with the embedding $H^{\frac{d}{2}^+} \hookrightarrow L^\infty$:
\begin{align*}
\left\vert \int F(t,z) -2Rz^2 \right\vert 
	& \leq \left\vert \int F(t,z)-2R^*z^2 \right\vert + \left\vert \int 2{R^*}z^2-2Rz^2 \right\vert \\
	& \leq \| (R^*-R)(t) \|_{L^\infty} \int z(t)^2 \leq Ce^{-\gamma_1 t} \| z \|_{L^2}^2.
\end{align*}

For now, we study the functional:
\begin{align*}
H(t):= \int \left( \vert \nabla z(t) \vert^2 -F(t,z) \right) \htilde(t) + z^2(t),
\end{align*}
and claim that
\begin{align}\label{claim}
\exists K>0, \forall t>T_0, \quad H(t) \leq K e^{-\gamma_1 t} \sup_{t'>t} \| z(t') \|_{H^1}^2.
\end{align}
It suffices to study the time variation of $H$, and find a lower bound by monotonicity. A computation gives us:
\begin{align}
\frac{d}{dt}H(t)
	& = \int \left( \vert \nabla z \vert^2 -F(t,z) \right) \frac{d\htilde}{dt} + 2 \int \partial_1 \Delta z \nabla z \cdot \nabla \htilde \label{1er_terme} \\
		& \quad + 2 \int \partial_1 \left( (z+R^*)^2-{R^*}^2 \right) \nabla z \cdot \nabla \htilde \label{2ieme_terme} \\
		& \quad - \int \left( \Delta z  + (z+R^*)^2 -{R^*}^2 \right)^2 \partial_1 \htilde \label{3ieme_terme} \\
		& \quad -2 \int \partial_1 \left( \Delta z + (z+R^*)^2-{R^*}^2 \right) z -2 \int {R^*}_t \left( (z+ R^*)^2 -{R^*}^2 -2{R^*}z \right) \htilde \label{4ieme_terme}
\end{align}
To find a lower bound of $\frac{dH}{dt}$, we notice first that the term (\ref{3ieme_terme}) is positive, because $h$ is decreasing in the first direction. We can now deal with (\ref{1er_terme}):
\begin{align*}
\MoveEqLeft
(\ref{1er_terme}) = \int \left( \vert \nabla z \vert^2 -F(t,z) \right) \frac{d\htilde}{dt} -2 \int \left( \vert \nabla \partial_1 z \vert^2 \right) \partial_1\htilde + \int \left( \partial_1 z \right)^2 \partial_1^3 \htilde \\
	& \geq \int \vert \nabla z \vert^2 \frac{d\htilde}{dt} + \int \left( \partial_1 z \right)^2 \partial_1^3 \htilde - \int \vert F(t,z)\vert\frac{d\htilde}{dt} 
\end{align*}

This is where the choice of the function $h$ is crucial. Because of the estimate on $\psit$, we have $\frac{d\htilde}{dt} \geq \frac{1}{4} \vert \partial_1^3 \htilde \vert$, and the sum of the two first terms is non-negative. We obtain, with $\| \htilde_1(t) R(t) \|_{L^\infty} \leq Ce^{-\gamma_1 t}$, the definition of $z$ and corollary (\ref{coro_multisol}):
\begin{align}
(\ref{1er_terme}) \geq -C \int \left( \vert z \vert^3(t) + \vert R^*(t) -R(t) \vert z^2(t) + \vert R(t) \vert z^2(t) \right) \frac{d\htilde}{dt}  \geq -Ce^{-\gamma_1 t}. \label{1er_terme_fin}
\end{align}

We now develop (\ref{2ieme_terme}):
\begin{align*}
(\ref{2ieme_terme}) = 2 \int z_1^2 z\htilde_1 +2 \int z_1^2 R^*\htilde_1  - \int z^2\htilde_{11},
\end{align*}
so the estimate is straightforward, with $ \| u -R \|_{L^\infty} \leq \| R^*-R \|_{H^d} \leq A_de^{-\delta t}$, and the embedding $H^2(\mathbb{R}^d) \hookrightarrow L^\infty$, or $ \| \partial_1(R^* -R) \|_{L^\infty} \leq A_{2+1}e^{-\delta t}$ :
\begin{align}
\vert ( \ref{2ieme_terme}) \vert & \leq Ce^{-\gamma_1 t} \| z_1 \|_{L^2}^2 + 2 \int z_1^2 \vert R \htilde_1 \vert + 2 \int z_1^2 \vert R^*-R \vert\vert \htilde_1 \vert - \int z^2 \vert R^*_1 \htilde_1 + R^*\htilde_{11} \vert \nonumber\\
	& \leq Ce^{-\gamma_1 t} \| z \|_{H^1}^2, \label{2ieme_terme_fin}
\end{align}
with $C$ depending on $A_1$, $A_2$, $A_3$ and $\gamma_1$ lower than $\delta$.

It remains the terms of (\ref{4ieme_terme}). We notice that $\int \partial_1 \Delta z z=\int \partial_1(z^3) =0$. The remaining terms give us:
\begin{align}
\left\vert (\ref{4ieme_terme})\right\vert 
	& = \left\vert 2 \int \left( \partial_1 R^* -R^*_t \htilde \right) z^2 \right\vert \leq 2 \| \partial_1 R^* -R^*_t \htilde \|_{L^\infty} \| z(t) \|_{L^2}^2 \nonumber\\
	& \leq 2 \left( \| \partial_1 \left( \Delta R^* +{R^*}^2- \Delta R - R^2 \right) \htilde \|_{L^\infty} + \| \partial_1 \left( \Delta R +R^2 \right) h-R_1 \|_{L^\infty} + \| \partial_1 (R-R^*) \|_{L^\infty} \right) \| z \|_{L^2}^2 \nonumber \\
	& \leq Ce^{-\gamma_1 t} \| z (t) \|_{L^2}^2, \label{4ieme_terme_fin}
\end{align}
with $C$ depending on $A_{3+3}$ and $0<\gamma_1<\delta$ due to the bound by the $H^{2+3}$-norm of the error. Notice that for this step it was necessary to get the $H^5$-norm of the error.

By summing up (\ref{1er_terme_fin}), (\ref{2ieme_terme_fin}) and (\ref{4ieme_terme_fin}), we obtain a lower bound on the time derivative of $H$: $\frac{dH}{dt}(t) \geq C e^{-\gamma_1 t} \| z(t) \|_{H^1}$. An integration from $t$ to $+\infty$ gives us (\ref{claim}).

\vspace{0.5cm}\textbf{End of Step 1}

By summing up the estimates (\ref{coercivite}), (\ref{coercivite2}), (\ref{coercivite3}) and (\ref{claim}), we obtain the desired estimate (\ref{principal}) for the term $\| \widetilde{z}(t) \|_{H^1}$.

\vspace{0.5cm}\textbf{Step 2.} Estimate on the $a^{ik}(t)$.

Let derive $\widetilde{z}$ along the time.
\begin{align*}
\frac{d}{dt} \left( \widetilde{z}(t) \right) = \frac{dz}{dt}(t) - \sum_{ik} \frac{da^{ik}}{dt}(t) R_i^k(t) + \sum_{ik} a^{ik}(t)R_{i1}^k(t)c^k.
\end{align*}
We can then express the linearised part of the equation of the error around $R^*$:
\begin{align*}
\partial_t \widetilde{z} + \left(\Delta \widetilde{z} + 2 R^* \widetilde{z} \right)_1 = - \sum_{ik} \frac{da^{ik}}{dt} R_i^k + \sum_{ik} a^{ik} \left( - \partial_1 \Delta R^k + R^k c^k -2 \partial_1(2R^*R^k) \right)_i +\partial_1 \left( 2R^*z -(R^*+z)^2 +{R^*}^2 \right).
\end{align*}
The scalar product of this expression with $R_{i_1}^{k_1}$ gives us, by reminding that $0 = \frac{d}{dt} \int \widetilde{z} R_i^k = \int \frac{d \widetilde{z}}{dt} R_i^k-c^k \int \widetilde{z} R_{i1}^k$:
\begin{align*}
\left\vert \frac{da^{i_1k_1}}{dt} \int (R_{i_1}^{k_1})^2 + \sum_{(i,k) \neq (i_1,k_1)} \frac{da^{i,k}}{dt} \int R_{i_1}^{k_1} R_i^k \right\vert
	& \lesssim \sum_{ik} \vert a^{ik} \vert e^{-\gamma_1t} + \| \widetilde{z}(t) \|_{H^1} + \| z(t) \|_{H^1}^2\\
	& \leq Ce^{-\gamma_1 t} \| z(t) \|_{H^1}
\end{align*}
where in the last inequality we used the step 1. One can notice that $\int R_{i_1}^{k_1}R_i^k$ is small while $(i,k)\neq (i_1,k_1)$. In particular, we find as in the modulation a matrix with a dominant diagonal, which implies that
\begin{align*}
\forall (i,k), \quad \left\vert \frac{da^{ik}}{dt} \right\vert \leq Ce^{-\gamma_1 t} \| z(t) \|_{H^1}.
\end{align*}
An integration from $t$ to $\infty$ gives us the estimate (\ref{principal}) for the coefficients $a^{ik}(t)$, and finishes the proof.

\vspace{0.5cm}\textbf{The critical case.} 

The estimates for (\gls{ZK23}) are similar, except that we add orthogonality conditions with respect to $Z^k$:

\begin{align*}
\widetilde{z}(t):= z(t) - \sum_{i=1}^d \sum_{k=1}^K a^{ik} R^k_i - \sum_{k=1}^K b^{k} Z^k,
\end{align*}
with
\begin{align*}
a^{ik} (t) := \frac{1}{\int \left(R^k_i(t) \right)^2dx} \int R^k_i(t) z(t) dx \text{ and } b^k (t) := \frac{1}{\int Z^k(t) ^2dx} \int Z^k(t) z(t) dx.
\end{align*}

With the estimate (\ref{coercivity}) on $\widetilde{z}(t)$, we obtain the lower bound asked in (\ref{principal}). The other arguments apply similarly. Once again, we use repetitively the exponential decay of $R^k_i$ and of $Z^k$.
\end{proof}

\appendix

\section{Coercivity}

We recall the lemma \ref{coercivity_lemm} of coercivity :

\begin{lemm*}\nonumber
For (\ref{ZK}), there exists $C_0 >0$, such that:
\begin{align*}
C_0 \| w \|_{H^1}^2 - \frac{1}{C_0} \sum_{k=1}^K \left( \int \widetilde{R^k}w \right)^2 \leq \sum_{k=1}^K \frac{1}{(c^k)^2} H_k(t).
\end{align*}
\end{lemm*}

\begin{proof}
Let deal first with the cases of (\gls{ZK22}) and (\gls{ZK32}). The proof is close to the step 2 of the appendix A of \cite{MM02}. We know a property close to coercivity of the operator $H_k(t)$: by the article of Weinstein, if we consider a function $v$ with the orthogonality conditions $v \perp Q$, and $\forall i \in \oned$, $v \perp \partial_i Q$, then the following operator is coercive:
\begin{align}\label{coercivite4}
\exists C_1 >0 \text{ , } \int (\nabla v) ^2 + v^2 - p Q^{p-1}v^2 \geq C_1 \| v \|_{H^1}^2.
\end{align}

We claim that in the case of the operators $(H_k(t))_k$, we obtain a similar coercivity condition. The cut-off function $\phi^k$ and the unsatisfied orthogonality condition $w \perp \widetilde{R^k}$ will change the previous estimate. Let the bilinear form $B_k(v,w):= \int \nabla v \cdot \nabla w +c^k v w -2 \widetilde{R^k} v w$. By developing the expression of $H_k(t)$:
\begin{align*}
H_k(t) = B_k(w \sqrt{\phi^k}, w\sqrt{\phi^k}) - \frac{1}{4}\int w^2 \frac{(\partial_{x_1}\phi^k)^2}{\phi^k}- \int w \partial_1w \partial_1 \phi^k.
\end{align*}
The first term can be developed, by decomposing $w\sqrt{\phi^k}(t):= \alpha_0 \widetilde{R^k} + \sum\limits_{i=1}^d \alpha_i \partial_i \widetilde{R^k} + w^k$, with
\begin{align*}
\alpha_0:= \frac{\langle w \sqrt{\phi^k}, \widetilde{R^k} \rangle}{\| Q_{c^k} \|_{L^2}^2} \text{ , } \alpha_i := \frac{\langle w \sqrt{\phi^k}, \partial_i \widetilde{R^k}\rangle}{\| \partial_i Q_{c^k}\|_{L^2}^2} \text{ , } \langle w^k , \widetilde{R^k} \rangle =\langle w^k , \partial_i \widetilde{R^k} \rangle=0,
\end{align*}
and using the Young inequality:
\begin{align*}
B_k(w \sqrt{\phi^k},w \sqrt{\phi^k}) & = B_k(w^k, w^k) + \alpha_0^2 B_k(\widetilde{R^k},\widetilde{R^k}) +2 \alpha_0 B_k(\widetilde{R^k}, w^k) \\
 & \geq \frac{C_2}{2}  \| w^k \|_{H^1}^2-\frac{1}{C_2} \left( \int \widetilde{R^k} w^k \sqrt{\phi^k} \right)^2,
\end{align*}
with $C_2>0$ a constant small enough. Furthermore, by the norm:
\begin{align*}
\| w^k \|_{H^1}  \geq \| w \sqrt{\phi^k} \|_{H^1} - C\sum_{i=0}^d \vert \alpha_i \vert \text{ so } \| w^k \|_{H^1}^2  \geq \frac{1}{2}\| w \sqrt{\phi^k} \|_{H^1}^2 - C \sum_{i=0}^d \vert \alpha_i \vert^2,
\end{align*}
we obtain the upper bound:
\begin{align*}
B_k(w \sqrt{\phi^k}, w \sqrt{\phi^k})  \geq \frac{C_2}{4} \| w \sqrt{\phi_k}\|_{H^1}^2- \frac{C}{C_2} \left( \left( \int \widetilde{R^k} w^k \sqrt{\phi^k} \right)^2 + \sum_{i=1}^d \left( \int \partial_i \widetilde{R^k} w^k \sqrt{\phi^k} \right)^2 \right).
\end{align*}
By the definition of $\phi^k$, the derivative with respect to $x_1$ will make a factor $\frac{1}{L}$ appear: $\left\vert \partial_1 \phi^k \right\vert\leq \frac{1}{L}\phi^k$, see (\ref{majoration_psi}), which implies by taking $L$ large enough depending on $C_2$ :
\begin{align*}
H_k(t) \geq \frac{C_2}{16}\int(\nabla w^2 + w^2) \phi_k - C \left( \left( \int \widetilde{R^k} w^k \sqrt{\phi^k} \right)^2 + \sum_{i=1}^d \left( \int \partial_i \widetilde{R^k} w^k \sqrt{\phi^k} \right)^2 \right).
\end{align*}
Let $k \in \llbracket 2, K-1\rrbracket$, and the interval $J_t^k:=\left[ \frac{1}{2}\left( x_1^k-c^k t-y_1^k+m^{k-1}(t) \right); \frac{1}{2}\left( x_1^k-c^k t-y_1^k+m^k(t) \right)\right]$. We obtain :
\begin{align*}
\left\vert \int \widetilde{R^k} w^k \sqrt{\phi^k} -\int \widetilde{R^k}w^k \right\vert
	& \lesssim \left( \int \left( \sqrt{\phi^k}-1 \right)^2 \mathbf{1}_{J_t^k} \right)^\frac{1}{2} \| w^k \|_{H^1}  + \left( \int \widetilde{R^k}^2 \mathbf{1}_{{J_t^k}^C} \right)^\frac{1}{2} \| w^k \|_{H^1} \\
 	& \lesssim \sqrt{L} e^{-\frac{1}{L} \frac{\sigma_0}{4}t}\| w ^k \|_{H^2} + \left( e^{-\frac{\sigma_0}{2}t} +\vert \widetilde{\bx^k}(t) -\widetilde{c^k}te_1 \vert \right) \| w^k \|_{H^1},
\end{align*}
and similarly:
\begin{align*}
  \left\vert \int \widetilde{R^k} w^k (\sqrt{\phi^k}-1) \right\vert + \sum_{i=1}^d \left\vert \int \partial_i \widetilde{R^k} w^k (\sqrt{\phi^k}-1) \right\vert \leq C \left(\sqrt{L}e^{-\frac{1}{L}\frac{\sigma_0}{4}t}  + \vert \widetilde{\bx^k}(t) - \widetilde{c^k} t\beone \vert \right) \|w^k \|_{H^1},
\end{align*}
with $C$ depending on the $(\by^k)_k$ and on the $(c^k)_k$. Adding the previous estimates from $1$ to $K$ with the weights $\frac{1}{(c^k)^2}$: 
\begin{align*}
\sum_{k=1}^K \frac{1}{(c^k)^2} H_k(t)
	& \geq \frac{C_2}{16(c^K)^2}\sum_{k=1}^K \int \left( \vert \nabla w \vert^2 +w^2 \right) \phi_k- \frac{C}{(c^1)^2} \left( \sqrt{L}e^{-\frac{1}{L}\frac{\sigma_0}{4}t} +\sum_k \vert \widetilde{ \bx^k}(t)- \widetilde{c^k}t\beone \vert \right) \| w \|_{H^1}^2 \\
		& \quad - C \sum_{k=1}^K \left( \left( \int \widetilde{R^k} w^k \right)^2 + \sum_{i=1}^d \left( \int \partial_i \widetilde{R^k} w^k \right)^2  \right).
\end{align*}
We find back the $H^1$-norm of the error on the all space by summing the $\phi_k$. Thus, we replace the last $w^k$ by there expression with $w\sqrt{\phi^k}$, the weak interactions between $\widetilde{R^k}$ and $(1-\sqrt{\phi^k})$, by the Young inequality and (\ref{proches2}), we obtain:
\begin{align}
\MoveEqLeft
\sum_{k=1}^K \frac{1}{(c^k)^2} H_k(t) \nonumber \\
	& \geq \frac{C_2}{16(c^K)^2} \| w \|_{H^1}^2 - \frac{C}{(c^1)^2} \left( \sqrt{L}e^{-\frac{1}{L}\frac{\sigma_0}{4}t} + e^{-\frac{\sigma_0}{2}t} +\alpha \right) \| w \|_{H^1}^2 - C \sum_{k=1}^K \left( \left( \int \widetilde{R^k} w \right)^2 + \sum_{i=1}^d \left( \int \partial_i \widetilde{R^k} w \right)^2  \right). \label{coercivite_pb_L}
\end{align}

By taking $L$ large enough depending only on $C_2$, $T_2$ large enough depennding on the different velocities, $\alpha$ small enough and the orthogonality conditions (\ref{ortho1}), there exists a constant $C_0$ small enough concluding the lemma for (\gls{ZK22}) and (\gls{ZK32}). 
 
Let us now deal with the case of (\gls{ZK23}). The result is similar, except that we asked for an other orthogonality condition (\ref{ortho2}), so we use an other parameter of modulation $(\widetilde{c^k}(t))_k$. Let $(\cdot, \cdot)$ denotes the $L^2$-scalar product. We claim, as in \cite{MM01} and \cite{Wei85}, that 
\begin{align}\label{debut_coercivite}
v\perp Q , v \perp Z \quad \Rightarrow  \quad (Lv,v)\geq 0.
\end{align}

We use a result of \cite{Wei85} concerning the operator $\Lc:=-\Delta +1 -pQ^{p-1}$:
\begin{align}\label{Weinstein}
\inf_{(v,Q)=0}(\Lc v,v)=0.
\end{align}
 Let suppose by contradiction that there exists $v$, satisfying $v\perp Z$, such that $(\Lc v,v)<0$. Let consider the operator $\Lc$ restricted to $\spann \{Z, v\}$. By definition of $Z$, $(\Lc (Z),Z)=-\lambda_0$. Thus we can find $v_0 \in \spann\{Z, v\}$, $v_0 \perp Q$ such that $(\Lc v_0,v_0)<0$, which contradicts (\ref{Weinstein}), and proves (\ref{debut_coercivite}). 

To obtain a similar coercivity inequality (\ref{coercivite4}), suppose by contradiction that:
\begin{align*}
0=\inf \left\{ (\Lc v,v) ; \| v \|_{L^2}=1, v \perp Z, v \perp \partial_1 Q, v \perp \partial_2 Q \right\}.
\end{align*}
By taking a sequence in $L^2$ for which the infimum is attained, up to a subsequence, it converges to an element $v$. It satisfies the orthogonality condition (\ref{ortho1}) and (\ref{ortho2}), and by rescaling arguments as in \cite{Wei85}, we can suppose that the norm of $v$ is $1$. The minimum is thus attained at a point $v\neq 0$, and there exists $(\alpha_0, \alpha_1, \alpha_2, \beta)$ among the critical points of the Lagrange multiplier problem:
\begin{align*}
\left\{
\begin{array}{l}
\Lc v -\beta v = \alpha_0 Z + \alpha_1 \partial_1 Q + \alpha_2 \partial_2 Q, \\
(\Lc v,v)=0, \\
v \perp Z,\quad  v \perp \partial_1 Q, \quad v \perp \partial_2 Q, \\
\| v \|_{L^2}=1.
\end{array}
\right.
\end{align*}
The scalar product with respect to $v$ gives us $\beta=0$, with the eigenvector $Z$ gives $\alpha_0=0$ and with $\partial_i Q$ gives $\alpha_i = 0$. This contradiction leads to a positive infimum, and the existence of a positive constant $C_1$ satisfying (\ref{coercivite4}).

The rest of the proof is then similar to the one of (\ref{ZK}).
\end{proof}

\section{Local well posedness}

Let us recall the local well posedness result that we will use.

\begin{theo}[{\cite{MP}, \cite{GH}, \cite{Kin} and \cite[Theorem 1.1]{LP09}}] \label{th:lwp}
Let $s > 1/2$. There exists a function $T: [0,+\infty) \to (0,+\infty)$ such that for any $u_0 \in H^s$, there exists a mild solution $u \in \mathcal C([-T(\| u_0 \|_{H^s}), T(\| u_0 \|_{H^s}), H^s)$  (to the Duhamel formulation) of \eqref{ZK}, which is furthermore unique in some subspace. The maximal interval of existence $(T_-(u_0),T_+(u_0))$ is open and does not depend on $s >1/4$, and if $T_+(u_0) <+\infty$, then $\| u(t) \|_{H^s} \to +\infty$ as $t \to T_+(u_0)$.

Furthermore, the flow $u_0 \mapsto u$ is continuous in the following sense: given $u_0 \in H^s$, then for any compact interval $J \subset (T_-(u_0),T_+(u_0))$ and $\e>0$, there exists $\delta >0$ such that if $\| v_0 - u_0 \|_{H^s} \le \delta$, then $J \subset (T_-(v_0),T_+(v_0))$ and $\sup_{t \in J} \| v(t) - u(t) \|_{H^s} \le \delta$ ($u$ and $v$ are the solutions to \eqref{ZK} with initial data $u_0$ and $v_0$ respectively.
\end{theo}

In particular, in the $L^2$ subcritical cases (\gls{ZK22}) and (\gls{ZK32}), if the initial condition is in $H^1$, one has global existence: $(T_-(u_0),T_+(u_0)) = \mathbb R$. The blow up criterion is relevant for our purpose only for the case of (\gls{ZK23}) which is $L^2$ critical.

\textbf{Acknowledgement.} The author would like to thank his advisor Raphaël Côte for suggesting this problem, for his constant scientific support, and whose pertinent comments greatly improved this manuscript.

\small

\bibliographystyle{plain}

\bibliography{biblio.bib}

\end{document}